\documentclass[reqno,a4paper]{amsart}

\usepackage{mathrsfs,amsmath,amssymb,graphicx}
\usepackage{mathtools}
\usepackage[unicode]{hyperref}
\usepackage{layout}
\usepackage[english]{babel} 
\usepackage{color}
\usepackage{fancyvrb}
\usepackage{enumitem} 
\usepackage{float}
\usepackage{multicol} 
\usepackage{subcaption}
\usepackage{graphicx}
\usepackage[percent]{overpic}

\usepackage{tikz}
\usetikzlibrary{calc}
\usetikzlibrary{matrix,arrows,decorations,chains,positioning}

\makeatletter
\@namedef{subjclassname@2020}{\textup{2020} Mathematics Subject Classification}
\makeatother

\theoremstyle{plain}
\newtheorem*{theorem*}{Theorem} 
\newtheorem{theorem}{Theorem}[section]
\newtheorem{lemma}[theorem]{Lemma}

\newtheorem{corollary}[theorem]{Corollary}

\theoremstyle{definition}
\newtheorem{definition}[theorem]{Definition}
\newtheorem{remark}[theorem]{Remark}

\newcommand{\pmobius}{{\mathsf M}}
\newcommand{\pklein}{{\mathsf K}}
\newcommand{\sdisk}{{\mathsf D}} 
\newcommand{\Disk}[1]{D_{\!#1}}
\newcommand{\disk}[1]{d_{\!#1}}
\newcommand{\xA}{x_{\!A}}

\newcommand{\geodisk}{{\mathfrak{D}}}
\newcommand{\picirc}{\pi_\circ}
\newcommand{\rmdisk}{\mathrm{D}}

\newcommand{\Sbb} {\mathbb S}
\newcommand{\sphere} {\Sbb^3}
\newcommand{\ball}{{\mathbf B}}
\newcommand{\asphere}{{\mathbf S}}

\newcommand{\HK}{{\rm HK}}
\newcommand{\HKp}{{\rm HK^+}}
\newcommand{\HKn}{{\rm HK^-}}
\newcommand{\emK}{{\rm K_{(l,m,n,p)}}}

\newcommand{\pair}{(\sphere,\HK)}
\newcommand{\pairp}{(\sphere, \HKp)} 
\newcommand{\pairn}{(\sphere, \HKn)}
\newcommand{\pairpind}[1]{(\sphere, \HK_{(#1)}^+)} 

\newcommand{\pairpn}{(\sphere, \HK^{\pm})}
\newcommand{\pairpnind}{(\sphere, \HK^{\pm}_{(l,m,n,p)})} 

\newcommand{\pairem}{{(\sphere,\emK)}}

\newcommand{\Ra}{\mathcal{R}}
\newcommand{\Ta}{\mathcal{T}}

\newcommand{\homeo}[1]{{\rm Homeo(#1)}}

\newcommand{\minus}{{\scalebox{.57}[.57]{$-$}}}
\newcommand{\plus}{{\scalebox{.57}[.57]{$+$}}}
\newcommand{\plusminus}{{\scalebox{.57}[.57]{$\pm$}}}
\newcommand{\minusplus}{{\scalebox{.57}[.57]{$\mp$}}}

\newcommand{\slope}{\mathtt{r}}

\newcommand{\ver}{\mathsf{v}}
\newcommand{\hor}{\mathsf{h}}

\newcommand{\algint}{\mathcal{I}_a} 
\newcommand{\geoint}{\mathcal{I}}

\newcommand{\Ce}{C_e}
\newcommand{\Co}{C_o}
 
\newcommand{\hopf}{{\mathbf{h}}}
\newcommand{\knot}{{\mathbf{k}}}
\newcommand{\link}{{\mathbf{l}}}

\newcommand{\charM}{{\Lambda_{M}}}
\newcommand{\charE}{{\Lambda_{\Compl \HK}}}
\newcommand{\annhk}{{\Lambda_{\textsc{hk}}}}

\newcommand{\Mo}{\mathtt{M}}
\newcommand{\Se}{\mathtt{S}}
\newcommand{\Kl}{\mathtt{K}}

\newcommand{\rnbhd}[1]{\mathfrak N(#1)}
\newcommand{\openrnbhd}[1]{\mathring{\mathfrak N}(#1)}
\newcommand{\Compl}[1]{E(#1)}
\newcommand{\ComplS}{\Compl S}
\newcommand{\cb}[1]{\vert \partial_f #1\vert}  
\newcommand{\gb}[1]{g(\partial #1)}  
\newcommand{\front}{\partial_f}

\newcommand{\bsimeq}{\simeq_\partial}

\newcommand{\eoQ}{\mathbb{Q}_o^e}
\newcommand{\pR}{{\mathbb{R}^2_\circ}} 
 
\newcommand{\MCG}{\mathrm{MCG}}
\newcommand{\Aut}[1]{\mathrm{Aut}(#1)}
\newcommand{\Ztwo}{\mathbb{Z}^2} 

\newcommand{\bt}{{\bf t}} 
\newcommand{\bm}{{\bf m}}
\newcommand{\bv}{{\bf v}}

\newcommand{\cout}[1]   {}

\numberwithin{equation}{section}

\title[Essential annuli]{Essential annuli in genus two handlebody exteriors}
 
\author{Yuya Koda, Makoto Ozawa, Yi-Sheng Wang}
\address{Department of Mathematics, Hiyoshi Campus, Keio University, 4-1-1, Hiyoshi, Kohoku, Yokohama, 223-8521, Japan~ \slash ~ 
International Institute for Sustainability with Knotted Chiral Meta Matter (WPI-SKCM$^2$), Hiroshima University, 1-3-1 Kagamiyama, Higashi-Hiroshima, 739-8526, Japan}
\email{koda@keio.jp}
\address{Department of Natural Sciences, Faculty of Arts and Sciences, Komazawa University, 1-23-1 Komazawa, Setagaya-ku, Tokyo, 154-8525, Japan}
\email{w3c@komazawa-u.ac.jp}
\address{National Sun Yat-sen University, Kaohsiung 804, Taiwan}
\email{yisheng@math.nsysu.edu.tw}

\date{\today}

\begin{document}
 
\subjclass[2020]{Primary 57M50; Secondary 57M15, 57K12}
\keywords{}
\thanks{Y. K. is supported by JSPS KAKENHI Grant Numbers JP20K03588, JP21H00978,  JP23H05437 and JP24K06744.
 M. O. is partially supported by Grant-in-Aid for Scientific Research (C) (No. 17K05262), The Ministry of Education, Culture, Sports, Science and Technology, Japan. 
Y-S. W. is supported by National Sun Yat-sen University and MoST (grant no. 110-2115-M-001-004-MY3), Taiwan.}

\begin{abstract}
We classify all potential configurations of essential annuli in a genus two atoroidal handlebody exterior in the $3$-sphere, building on two recent classifications: the classification of the JSJ-graph of the 
exterior and the classification of 
essential annuli in the exterior. 
In contrast to knots, 
genus two handlebody exteriors
may contain infinitely many non-isotopic essential annuli, due to the JSJ-graph classification.
Our main result characterizes the numbers of different types of essential annuli in such an infinite family.  
\end{abstract}
\maketitle
 
\section{Introduction}\label{sec:intro}

Essential surfaces of non-negative 
Euler characteristic play an essential
role in $3$-manifold topology; many theorems of fundamental importance, 
such as the prime decomposition by Kneser and Milnor \cite{Kne:29}, \cite{Mil:62}, the characteristic compression body by Bonahon \cite{Bon:83}, and 
the JSJ-decomposition by Jaco-Shalen  \cite{JacSha:79} and Johannson \cite{Joh:79}, and Thurston's hyperbolization \cite{Thu:82}, are based on the existence and non-existence of such surfaces. 
They play a crucial part not only in 
$3$-manifold classification, but also in the study of $3$-manifold mapping class groups---notably, the Dehn subgroup conjecture, now proved for all orientable $3$-manifolds by Hong-McCullough \cite{HonMcC:13}, asserts that Dehn twists along essential surfaces of non-negative Euler characteristic generate a finite index subgroup in the mapping class group (see Johannson \cite{Joh:79}). 
 
Given a submanifold $V$ of an orientable 
$3$-manifold $M$, the mapping class group $\MCG(M,V)$ of the pair $(M,V)$ is known as the \emph{Goeritz group} when $V$ and $\overline{M-V}$ are both handlebodies, whereas in the case $M$ is the $3$-sphere $\sphere$, it is often called the \emph{symmetry group} of the embedding $V\subset \sphere$. When $M=\sphere$ and $V$ is a union of finitely many disjoint solid tori, or equivalently, $(\sphere,V)$ is a link, the symmetry group $\MCG(\sphere,V)$ 
has been extensively studied, and its structure has been determined for a great numbers of links; see Kawauchi \cite[Chapter $10$]{Kaw:96} and references therein. 
For a general $V$, in view of the Dehn subgroup conjecture, the symmetry group $\MCG(\sphere,V)$ is largely governed by essential surfaces of non-negative Euler characteristic in the exterior $\Compl V$ of $V$ in $\sphere$, and how their boundary behaves in relation to $V$; as yet though not much is understood about the symmetry group structure in general.

When $V=\HK$ is a genus two handlebody, the pair $\pair$ is called 
a genus two handlebody-knot. 
Due to 
Scharlemann \cite{Sch:04}, Akbas \cite{Akb:08}, Cho \cite{Cho:08} and the first-named author \cite{Kod:15}, 
it is known that the symmetry group $\MCG(\sphere,\HK)$ is finitely presented if $\Compl\HK:=\overline{\sphere-\HK}$ is $\partial$-reducible. 
On the other hand, based on 
the boundary behavior of essential annuli 
in relation to $\HK$, the first 
two authors \cite{KodOzaGor:15} classify essential annuli in $\Compl\HK$ into four groups, which can be further divided into ten types. 
The classification allows more systematic study of the symmetry group structure; for instance, symmetry groups of several classes of handlebody-knots with a unique essential annulus are computed in \cite{Wan:21} and \cite{Wan:23} by the third-named author. 
Further, making use of the classification, Funayoshi and the first-named author \cite{FunKod:20} obtain a finiteness result: the symmetry group $\MCG(\sphere,\HK)$ is finite if and only if
$\pair$ is atoroidal, namely, $\Compl \HK$ containing no essential tori.

While it is well-known that 
the symmetry group  
of a non-satellite knot is cyclic or dihedral; see again Kawauchi \cite{Kaw:96}), no general classification as such is known for the symmetry group of an atoroidal genus two handlebody-knot $\pair$. 
However, since the symmetry group 
$\MCG\pair$ is finite, it is expected that the group structure depends 
strongly on how characteristic annuli in $\Compl\HK$ are configured, and the configuration can be encoded in the \emph{JSJ-graph}, the dual graph of the JSJ decomposition of $\Compl\HK$.     

The JSJ-graph of the exterior of a genus two atoroidal handlebody-knot $\pair$ is classified into fourteen types in \cite{Wan:22p}\footnote{where JSJ-graph is called \emph{characteristic diagram}.}, where, 
combining with the annulus classification in \cite{KodOzaGor:15}, the third-named author investigates configuration of  
non-separating essential annuli in $\Compl\HK$. The investigation leads to some structural results on $\MCG\pair$ in the case $\Compl\HK$ 
admits a non-separating essential annulus. The JSJ-graph classification also implies, in contrast to knots, a genus two handlebody-knot exterior may admit infinitely many essential annuli, and if so, all but only one of them are separating.   

Motivated by this, the present paper 
examines separating essential annuli in 
an atoroidal genus two handlebody-knot 
exterior;
we show that when 
the exterior admits infinitely many essential annuli, all but finitely many of them are of a type in \cite{KodOzaGor:15} given by non-integral Dehn surgery on hyperbolic knots. Particularly, each such handlebody-knot gives rise to an infinite family of Eudave-Mu\~noz knots \cite{Eud:02}. 

In addition, we summarize, to our knowledge, all known results on how various types of essential annuli can be configured in a genus two atoroidal handlebody-knot exterior as a basis for further research on the symmetry group structure. This piece of information is packed in an enhanced JSJ-graph, called the \emph{relative JSJ-graph}. 
The relative JSJ-graph has also been employed to study the Gordon-Luecke problem---to what extent the handlebody exterior determines the handlebody-knot. In general, 
a genus two handlebody-knot is not determined by its exterior; examples are given in Motto \cite{Mott:90}, Lee-Lee \cite{LeeLee:12}, and Bellettini-Paolini-Wang \cite{BelPaoWan:20a}. However, the third-named author found in \cite{Wan:23p}\footnote{where relative JSJ-graph is called \emph{annulus diagram}.} that, for genus two atoroidal handlebody-knots with certain relative JSJ-graphs, the exterior does determine the handlebody-knot.

\section{Preliminaries and Main results}
\subsection*{JSJ-graph}
Recall that the JSJ-decomposition
asserts, 
for every irreducible, $\partial$-irreducible, compact, orientable $3$-manifold, there exists a surface $S$, called \emph{characteristic surface}, unique up to isotopy, consisting of essential annuli and tori such that,
first, for every component $X$ in the exterior $\Compl S:=\overline{M-\rnbhd{S}}$, either $X$ can be \emph{admissibly I-/Seifert fibered}, that is, $X$ fibered so that its frontier 
$\partial_f X$ is a union of fibers, or $X$ is \emph{simple}, namely, every essential annulus of $M$ in $X$ being isotopic to a component of $\partial_f X$, where $\rnbhd{S}$ is a regular neighborhood of 
$S$. An essential annulus or torus 
is \emph{characteristic} if it is
isotopic to a component of $S$, and is \emph{non-characteristic} otherwise.

To encode configuration of admissibly I-/Seifert fibered and simple components in 
$M$, we define 
the \emph{JSJ-graph} $\charM$ as follows: 
Assign a node to each component in $\Compl S$, and to each component $N$ of $\rnbhd{S}$, 
we assign an edge with adjacent node(s) corresponding to 
the component(s) of $\Compl S$ meeting $N$. 
To distinguish I-fibered, Seifert fibered, and simple components, we use filled squares, filled circles, and 
hollow circles for nodes representing them, respectively. 

\subsection*{Handlebody-knots}
A genus $g$ handlebody-knot $\pair$ 
is a genus $g$ handlebody $\HK$ 
in $\sphere$. The genus one handlebody-knot theory is equivalent to the study of classical knots. 
The JSJ-graph of $\pair$ is defined
to be the JSJ-graph $\charE$ of its exterior $\Compl\HK$.
The JSJ-graph of a non-satellite, non-trivial knot is rather simple: it is either \includegraphics[scale=.2]{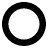} or
$\blacksquare$, which corresponds to 
a torus or hyperbolic knot, respectively.
In particular, a non-satellite knot exterior admits no characteristic annulus, and contains at most one non-characteristic annulus. 
By comparison, there are fourteen types of JSJ-graphs for genus two handlebody-knots as classified in Fig.\ \ref{tab:char_daigram}, and their exteriors may contain infinitely many non-characteristic annuli. The following 
result on number of characteristic and non-characteristic annuli from \cite{Wan:22p}; see Section \ref{sec:number} for a recollection.

\begin{theorem}{\cite[Theorem $1.1$ and Corollary $1.3$]{Wan:22p}}\label{intro:teo:number} 
Let $\pair$ be a non-trivial atoroidal genus two handlebody-knot.
Then $\Compl\HK$ admits at most three characteristic annuli,
and in addition,   
\begin{enumerate}[label=\textnormal{(\roman*)}]
\item it admits two non-characteristic annuli 
if its JSJ-graph
is one of the following
\begin{equation}\label{intro:two_nonchar}
\raisebox{.2cm}{
\includegraphics[scale=.18]{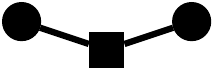} 
}
\quad,\quad
\includegraphics[scale=.18]{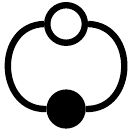}\quad,\quad
\includegraphics[scale=.18]{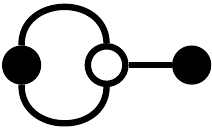}
\quad
,
\quad
\includegraphics[scale=.18]{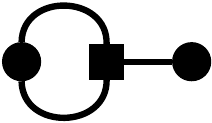}\quad
; 
\end{equation}
\item it admits infinitely many non-characteristic annuli 
if its JSJ-graph
is 
\includegraphics[scale=.18]{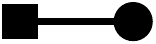};
\item no non-characteristic annuli exist in $\Compl\HK$ otherwise.
\end{enumerate}
\end{theorem}    
Note that the trivial graph Fig.\ \ref{fig:000h} corresponds to 
the case where $\Compl\HK$ admits a complete hyperbolic 
structure with totally geodesics boundary by Thurston's hyperbolization theorem. 
The JSJ-graph depends only on $\Compl\HK$, yet 
inequivalent handlebody-knots with homeomorphic exteriors abound. To capture missing information, 
the boundary behavior of characteristic annuli in relation to $\HK$ needs to be taken into account. 
\begin{figure}[t]
\begin{subfigure}{0.24\textwidth}
\centering
\includegraphics[scale=.3]{000h}
\caption{
}
\label{fig:000h}  
\end{subfigure}
\begin{subfigure}{0.24\textwidth}
\centering
\includegraphics[scale=.3]{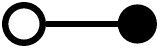}
\caption{
}
\label{fig:100h}  
\end{subfigure}
\begin{subfigure}{0.24\textwidth}
\centering
\includegraphics[scale=.3]{100i}
\caption{
}
\label{fig:100i}  
\end{subfigure}
%
\begin{subfigure}{0.24\textwidth}
\centering
\includegraphics[scale=.3]{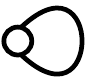}
\caption{}
\label{fig:110h}  
\end{subfigure}
\bigskip
\begin{subfigure}{0.24\textwidth}
\centering
\includegraphics[scale=.3]{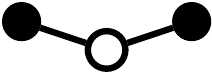}
\caption{}
\label{fig:200h}  
\end{subfigure}
\begin{subfigure}{0.24\textwidth}
\centering
\includegraphics[scale=.3]{200i}
\caption{}
\label{fig:200i}  
\end{subfigure}
%
%
\begin{subfigure}{0.24\textwidth}
\centering
\includegraphics[scale=.3]{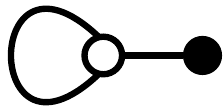}
\caption{}
\label{fig:210h}  
\end{subfigure}
\begin{subfigure}{0.24\textwidth}
\centering
\includegraphics[scale=.3]{201h}
\caption{}
\label{fig:201h}  
\end{subfigure}
\bigskip
\begin{subfigure}{0.24\textwidth}
\centering
\includegraphics[scale=.3]{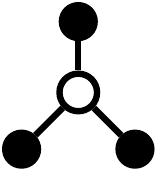}
\caption{}
\label{fig:300h}  
\end{subfigure}
%
%
\begin{subfigure}{0.24\textwidth}
\centering
\includegraphics[scale=.3]{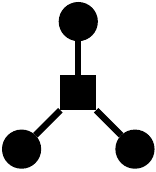}
\caption{}
\label{fig:300i}  
\end{subfigure}
\begin{subfigure}{0.24\textwidth}
\centering
\includegraphics[scale=.3]{301h}
\caption{}
\label{fig:301h}  
\end{subfigure}
\begin{subfigure}{0.24\textwidth}
\centering
\includegraphics[scale=.3]{301i}
\caption{}
\label{fig:301i}  
\end{subfigure}
\bigskip
\begin{subfigure}{0.24\textwidth}
\centering
\includegraphics[scale=.33]{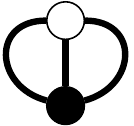}
\caption{}
\label{fig:303h}  
\end{subfigure}
\begin{subfigure}{0.3\textwidth}
\centering
\includegraphics[scale=.33]{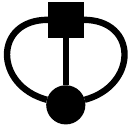}
\caption{}
\label{fig:303i}  
\end{subfigure}
\caption{Table of JSJ-graphs.
}
\label{tab:char_daigram}
\end{figure}
\subsection*{Koda-Ozawa Classification}
While an essential annulus in a knot exterior is either cabling or decomposing, there are up to ten types of 
essential annuli in a genus two handlebody-knot exterior, based on the classification in 
\cite{KodOzaGor:15} and \cite{FunKod:20}.    
Let $A$ be an essential annulus in $\Compl\HK$. 
The annulus $A$ is of \emph{type $1$} if both components of $\partial A$ bound disks in $\HK$; the existence of a type $1$ annulus implies the toroidality of $\pair$. 
The annulus $A$ is of \emph{type $2$} if exactly one component of 
$\partial A$ bounds a disk $D$ in $\HK$, 
and further $A$ is said to be of 
\emph{type $2$-$1$} if $D$ is non-separating and of \emph{type $2$-$2$} otherwise. The symbols $\hopf_\ast$ are reserved for type $2$-$\ast$ annulus, $\ast=1,2$.

The annulus $A$ is of \emph{type $3$} if 
no boundary components of $\partial A$ 
bounds disks in $\HK$, but $\partial \HK$ 
admits a compression disk $D$ in $\sphere$ 
disjoint from $A$. 
If $D$ is in the exterior $\Compl\HK$, 
then $A$ is of \emph{type $3$-$1$}. 
If $D$ is in the handlebody $\HK$, 
then there are two possibilities: 
$D$ does not separate components of $\partial A$
or it does; $A$ is said to be of
\emph{type $3$-$2$} (resp.\ \emph{type $3$-$3$}) if it is the former (resp.\ the latter); further, components of $\partial A$ are parallel (resp.\ non-parallel) if $A$ is of type $3$-$2$ (resp.\ type $3$-$3$). 
The existence of type $3$-$1$ 
annulus implies the reducibility and hence toroidality of $\pair$ \cite[Lemma $2.24$]{Wan:22p}.
On the other hand, by \cite[Lemmas $2.1$, $2.3$]{FunKod:20}, 
if $A$ is of type $3$-$2$ (resp.\ type $3$-$3$) annulus, then there exists 
a unique essential non-separating (resp.\ separating) disk $D\subset \HK$ disjoint from $\partial A$. This allows us to further divide type $3$-$2$ (resp.\ type $3$-$3$) annuli in two families: if 
$A$ is essential in the exterior of $\overline{\HK-\rnbhd{D}}$,
then it is of \emph{type $3$-$2$i (resp.\ $3$-$3$i)}, and is of \emph{type $3$-$2$ii (resp.\ $3$-$3$ii)} otherwise. The notation 
$\knot_\ast$ (resp.\ $\link_\ast$), $\ast=1,2$, 
is reserved for 
annuli of type $3$-$2\star$ (resp.\ type $3$-$3\star$), 
$\star=$i,ii, respectively.\footnote{In \cite[Section $4$]{KodOzaGor:15}, the type of M\"obius bands corresponding to type $3$-$2$ii annuli is missing. In \cite[Proof of Theorem $4.1$]{KodOzaGor:15}, this case occurs 
when $P'$ can be $\partial$-compressed onto $A$ in $\Compl Y$.}

Lastly, $A$ is of \emph{type $4$} if $\partial A$ is parallel in $\partial \HK$, and there is no compressing disk 
of $\partial\HK$ in $\sphere$ disjoint from $A$; $A$ is of \emph{type $4$-$1$} if $\pair$ 
is toroidal, and is of \emph{type $4$-$2$}
otherwise. Type $4$ annuli are intimately linked to non-integral toroidal Dehn surgery. Its original definition \cite[Section $3$]{KodOzaGor:15}, in fact, is phrased completely in terms of Eudave-Mu\~noz knots.

By the definition, when $\pair$ is atoroidal, only 
seven out of the ten types can exist, namely, types $2$-$i$, $i=1,2$,
types $3$-$2\star$, $3$-$3\star$, $\star=i,ii$, and type $4$-$1$. 
In Sections \ref{subsec:emknots} and \ref{subsec:fourone_classification}, we investigate type $4$-$1$ annuli via Eudave-Mu\~noz's tangles \cite{Eud:02}, 
and prove the following classification result. 
\begin{theorem}\label{intro:teo:hk_w_fourone}
If 
the exterior of an atoroidal genus two handlebody-knot $\pair$ admits a type $4$-$1$ annulus $A$, then $A$ is non-characteristic
and the JSJ-graph of $\pair$ is either \includegraphics[scale=.25]{100i}
or 
\includegraphics[scale=.25]{200i}.
\end{theorem} 
Section \ref{sec:nonchar} considers the opposite: if an atoroidal genus two handlebody-knot exterior
admit non-characteristic annuli, 
how many of them are of type $4$-$1$?

\subsection*{Relative JSJ-graph}
\begin{figure}[t] 
\begin{subfigure}{0.24\textwidth}
\centering
\includegraphics[scale=.35]{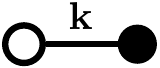}
\caption{
}
\label{fig:a100h}  
\end{subfigure}
\begin{subfigure}{0.24\textwidth}
\centering
\includegraphics[scale=.35]{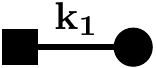}
\caption{
}
\label{fig:a100i}  
\end{subfigure}
%
\begin{subfigure}{0.24\textwidth}
\centering
\includegraphics[scale=.33]{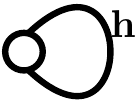}
\caption{}
\label{fig:a110h_hopf}  
\end{subfigure}
\bigskip
\begin{subfigure}{0.24\textwidth}
\centering
\includegraphics[scale=.33]{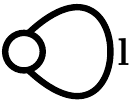}
\caption{}
\label{fig:a110h_link}  
\end{subfigure}
\begin{subfigure}{0.24\textwidth}
\centering
\includegraphics[scale=.35]{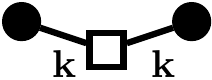}
\caption{}
\label{fig:a200hi}  
\end{subfigure}
%
%
%
\begin{subfigure}{0.24\textwidth}
\centering
\includegraphics[scale=.3]{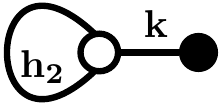}
\caption{}
\label{fig:a210h.hopf}  
\end{subfigure}
\begin{subfigure}{0.24\textwidth}
\centering
\includegraphics[scale=.3]{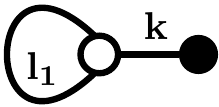}
\caption{}
\label{fig:a210h.link}  
\end{subfigure}
\begin{subfigure}{0.24\textwidth}
\centering
\includegraphics[scale=.28]{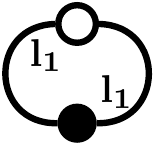}
\caption{}
\label{fig:a201h}  
\end{subfigure}

%
\begin{subfigure}{0.24\textwidth}
\centering
\includegraphics[scale=.33]{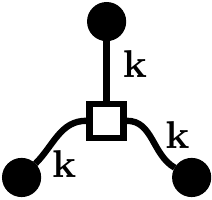}
\caption{}
\label{fig:a300hi}  
\end{subfigure}
%
%
\begin{subfigure}{0.24\textwidth}
\centering
\includegraphics[scale=.38]{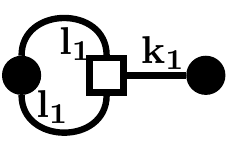}
\caption{}
\label{fig:a301hi}  
\end{subfigure}
\begin{subfigure}{0.24\textwidth}
\centering
\includegraphics[scale=.38]{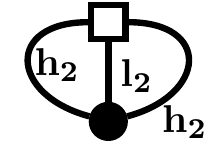}
\caption{}
\label{fig:a303hi}  
\end{subfigure}
\caption{Table of relative JSJ-graphs: the label $\hopf$ stands for 
either $\hopf_1$ or $\hopf_2$; that is, two possibilities may occur. The same applies to the labels $\knot$ and $\link$; the hollow square $\square$, likewise, means the node is either \includegraphics[scale=.2]{000h} or $\blacksquare$. For instance, Fig.\ \ref{fig:a300hi} alone accounts for eight different types.}
\label{tab:ann_daigram}
\end{figure} 
Given a non-trivial atoroidal genus two handlebody-knot $\pair$ with the exterior $\Compl\HK$ admitting 
an essential annulus, that is, excluding Fig.\ \ref{fig:000h}, 
the \emph{relative JSJ-graph} $\annhk$ of $\pair$ is defined as
the JSJ-graph $\charE$ with each edge labeled with the symbol corresponding to the type of the annulus the edge represents. 
%
Section \ref{subsec:classification} summarizes known results about the relative JSJ-graph, and gives
the following classification theorem.
 
\begin{theorem}\label{intro:teo:ann_diag}
Relative JSJ-graphs are classified into $30$ types in Fig.\ \ref{tab:ann_daigram}. 
\end{theorem} 
 
\begin{remark}
For some relative JSJ-graphs, there are currently no known examples of handlebody-knots realizing them. 
Notably, it is unclear whether all possible types in Fig.\ \ref{fig:a300hi} occur; also, no handlebody-knots realizing 
the two possible relative JSJ-graphs in Fig.\ \ref{fig:a301hi} is known. 
The former is closely related to $3$-punctured spheres in a $3$-component link exterior (see Eudave-Ozawa \cite{EudOza:19}), and 
the latter related to once-punctured Klein bottles with non-integral slope in a knot exterior and $3$-punctured spheres in a $2$-component link exterior.  
\end{remark} 
 
\subsection*{Non-characteristic annuli}
By Theorems \ref{intro:teo:number} and 
\ref{intro:teo:ann_diag},
given an atoroidal genus two handlebody-knot $\pair$,  
the exterior  
admits a non-characteristic annulus if and only if 
its relative JSJ-graph is one of the following:
\begin{equation}\label{eq:ann_diag_nonchar}
\includegraphics[scale=.25]{r100i},\quad 
\includegraphics[scale=.25]{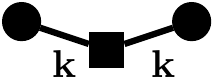},\quad
\includegraphics[scale=.2]{r201h},\quad 
\includegraphics[scale=.27]{r301hi}.
\end{equation} 
Furthermore, $\Compl\HK$ admits infinitely many non-characteristic annuli if $\annhk$ is \includegraphics[scale=.23]{r100i}, and admits exactly two non-characteristic annuli 
in the other three cases. 
Additionally, if $\Compl\HK$ admits 
a type $4$-$1$ annulus, then 
$\annhk$ is \includegraphics[scale=.23]{r100i} or \raisebox{-.3\height}{\includegraphics[scale=.23]{r200i}}. 
Section \ref{sec:nonchar} investigates
types of non-characteristic annuli   
and proves Theorems \ref{teo:typeM}, \ref{teo:typeK}, which we summarize as follows. 
\begin{theorem}\label{intro:teo:nonchar}\hfill
\begin{enumerate}[label=\textnormal{(\roman*)}]
\item\label{intro:itm:infinite} If $\Compl\HK$ admits infinitely many non-characteristic annuli (\includegraphics[scale=.2]{r100i}), then all but at most five of them are of type $4$-$1$.  
\item\label{intro:itm:two} 
If $\Compl\HK$ admits two non-characteristic annuli, then at most one of them is of type $4$-$1$.
\end{enumerate}   
\end{theorem} 
The above upper bounds are sharp; particularly,
the handlebody-knot $5_2$ in the table of Ishii-Kishimoto-Moriuchi-Suzuki \cite{IshKisMorSuz:12} attains the upper bound in Theorem \ref{intro:teo:nonchar}\ref{intro:itm:infinite}; see Section \ref{subsubsec:example}. 
The two non-characteristic annuli of 
type $3$-$2$ in the latter two cases of \eqref{eq:ann_diag_nonchar}
are classified in Theorem \ref{teo:typeS}, based on the slope of the type $3$-$3$ annuli that correspond to the edges with the label $\link_1$. 

\subsection*{Convention}
We work in the piecewise linear category.
Given a subpolyhedron $X$ of a manifold $M$, we denote by $\overline{X}$, 
$\mathring{X}$, $\front X$ and $\mathfrak{N}(X)$, 
the closure, the interior, the frontier, and a regular neighborhood of $X$ in $M$, respectively.
The \emph{exterior} $\Compl X$ of $X$ in $M$ is defined to be 
the complement of $\openrnbhd{X}$ in $M$ 
if $X\subset M$ is of positive codimension or the closure of $M-X$ otherwise,. 
Also, $\vert X\vert$ stands for the number of components in $X$.

Submanifolds of $M$ are assumed to be proper and in general position.
Given a loop (resp.\ based loop $l$) in
$M$, we use the notation $[l]$ 
to denote the homology (resp.\ homotopy) class it induces in $H_1(M)$ (resp.\ $\pi_1(M)$). 
We shall use the same notation to denote a \emph{path} in $M$ and its image, and juxtaposition for the \emph{path composition} and $\simeq$ 
for \emph{homotopic with endpoints fixed}.

Throughout the paper, the pair 
$(\sphere,K)$ denotes an embedding of a space $K$ in $\sphere$. By $\pair$, 
we understand a non-trivial, \emph{atoroidal, genus two} handlebody-knot, and $\Compl\HK$ denotes its exterior, and $\charE$ and $\annhk$ its JSJ-graph and relative JSJ-graph, respectively. 
Given a surface $S\subset\Compl\HK$ and a component $X$ in the exterior $\Compl S$ of $S\subset\Compl\HK$, by $X$ is I-/Seifert fibered,
we understand $X$ is \emph{admissibly} I-/Seifert fibered
in $\Compl\HK$. Given a solid torus $V$, $cV$ stands for the core of $V$, and an \emph{$n$-punctured surface} here is a surface with
$n$ disjoint \emph{open disks} removed.


\section{Number of essential annuli}\label{sec:number}
Here we recall how Theorem \ref{intro:teo:number}
and the table in Fig.\ \ref{tab:char_daigram} are 
derived from results in \cite{Wan:22p}. They allow us to classify handlebody-knots whose exteriors admit a non-characteristic annulus 
into three categories.

Throughout the section, $S$ is the characteristic surface
of $\Compl\HK$ and $\ComplS$ the exterior of $S$ in $\Compl\HK$. 
Given a component $X$ of $\ComplS$, we 
set $\partial_b X:=\HK\cap X=\overline{\partial X-\partial_f X}$.

We remark first that, the table in \cite[Figure $1$]{Wan:22p} does not distinguish I-fibered components from Seifert fibered ones, yet this piece of information can easily be filled in by \cite[Proposition $2.21$(i), (iv)]{Wan:22p}, which asserts that
$\Compl\HK$ contains a unique component 
$X$ with the genus $\gb X=2$, and it is either
I-fibered or simple, and 
every other component $Y$ has $\gb Y=1$, and is adjacent to $X$. Thus, every $\charE$ 
has exactly one filled square or hollow circle with every edge adjacent to it.

Furthermore, it is shown in \cite[Proposition $2.21$(ii)]{Wan:22p} that every component $Y$ with $\gb Y=1$ is a Seifert fibered solid torus, and it follows from Fig.\ \ref{tab:char_daigram} that the number $\cb Y\leq 3$. Additionally, $Y$ has an exceptional fiber if and only if $\cb Y=1$ or $2$ by \cite[Theorem $3.14$]{Wan:22p}.

On the other hand,
by \cite[Proposition $2.21$(v)]{Wan:22p}, 
an I-fibered component is I-fibered
over a pair of pants, or over a punctured M\"obius band, or over a punctured Klein bottle.    

\begin{definition}
Given $X$ a component of $\ComplS$,
then an annulus $A\subset X$
is \emph{admissible} 
if $\partial A\subset \partial_b X$.

An admissible annulus $A\subset X$ is \emph{essential} 
if there exists no disk in $D\subset X$ 
such that $D\cap(\partial_b X\cup A)=\partial D$ 
with $D\cap A\subset A$ an essential arc or circle, and 
is \emph{$\partial_f$-parallel} if $A$ and a component of $\partial_f X$ is isotopic in $\Compl\HK$ 
via an isotopy in $X$.  
\end{definition}

By the engulfing property \cite[Corollary $10.10$]{Joh:79}, 
every non-characteristic annulus in $\Compl\HK$ 
is isotopic to a non-$\front$-parallel 
essential annulus in an I-/Seifert fibered component $X\subset\ComplS$. Conversely, every non-$\front$-parallel essential annulus in an I-/Seifert fibered component $X$ in $\ComplS$ is a non-characteristic annulus in $\Compl\HK$. Also, since 
no two components in $S$ are parallel in $\Compl\HK$, two admissible, non-$\partial_f$-parallel essential annuli are isotopic in $X$ 
if and only if they are isotopic in $\Compl\HK$.

%
On the other hand, by the vertical-horizontal theorem \cite[Proposition $5.6$]{Joh:79}, 
a component $X\subset\ComplS$ 
admits a non-$\front$-parallel, essential admissible annulus
if and only if $X$ is I-fibered over a once-punctured 
Klein bottle or over a once-punctured M\"obius band, or 
$X$ is a Seifert fibered solid torus with $\cb X=2$. 
Going through each item in Fig.\ \ref{tab:char_daigram}, we see that
the three cases are mutually exclusive:  
\begin{lemma}\hfill\label{lm:typeKMS}
\begin{enumerate}[label=\textnormal{(\roman*)}]
\item\label{itm:typeK} $\ComplS$ admits an 
I-fibered component over a 
once-punctured Klein bottle if and only if $\charE$ is 
\includegraphics[scale=.18]{100i};
\item\label{itm:typeM} $\ComplS$ admits an 
I-fibered component over a once-punctured 
M\"obius band if and only if $\charE$ is 
\includegraphics[scale=.18]{200i}; 
\item\label{itm:typeS} $\ComplS$ admits a Seifert fibered component $X$  
with $\cb X=2$ if and only if it is 
\raisebox{-.3\height}{\includegraphics[scale=.16]{201h}}
or \raisebox{-.35\height}{\includegraphics[scale=.2]{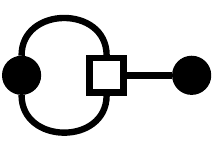}},
where 
\raisebox{-.1\height}{\includegraphics[scale=.23]{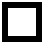}} is $\blacksquare$ or \includegraphics[scale=.2]{000h}.
\end{enumerate}
\end{lemma}
 
\begin{definition}
$\pair$ is said to be of type $\Kl$, $\Mo$ or $\Se$, 
if it is the case \ref{itm:typeK}, \ref{itm:typeM} or \ref{itm:typeS} in Lemma \ref{lm:typeKMS}, respectively.
\end{definition} 
 
\subsection{Type $\Kl$}\label{subsec:typeK}
Let $\pair$ be of type $\Kl$, and 
$X\subset \Compl\HK$ 
the I-fibered component. Denote by 
$\pi:X\rightarrow \pklein$ the bundle projection
over a once-punctured Klein bottle $\pklein$. Then 
$\Compl\HK$ admits infinitely many non-characteristic annuli, up to isotopy, since 
$X$ admits infinitely many essential, non-$\front$-parallel admissible annuli. 
Moreover, exactly one of them is non-separating.
By the horizontal-vertical theorem, these 
non-characteristic annuli can be obtained as follows: 
Choose two oriented simple loops $\alpha,\beta\subset \pklein$
as in Fig.\ \ref{fig:pklein}, 
and let $\gamma_n$ be a simple loop homotopic 
to $\alpha\beta^n$ (see Fig.\ \ref{fig:pklein_eg}). 
Then $M_n:=\pi^{-1}(\gamma_n)$ 
is an essential M\"obius band in $\Compl\HK$.
Denote by $A_n$ the frontier of a regular neighborhood 
of $M_n\subset \Compl\HK$. Then $\{A_n\}_{ n\in\mathbb{Z}}$ gives us all the \emph{separating}, non-characteristic annuli in $\Compl\HK$; the unique \emph{non-separating} annulus in
$\Compl\HK$ is given by the preimage $\pi^{-1}(\beta).$

\begin{figure}[b]
\begin{subfigure}{.48\linewidth}
\centering
\begin{overpic}[scale=.18,percent]{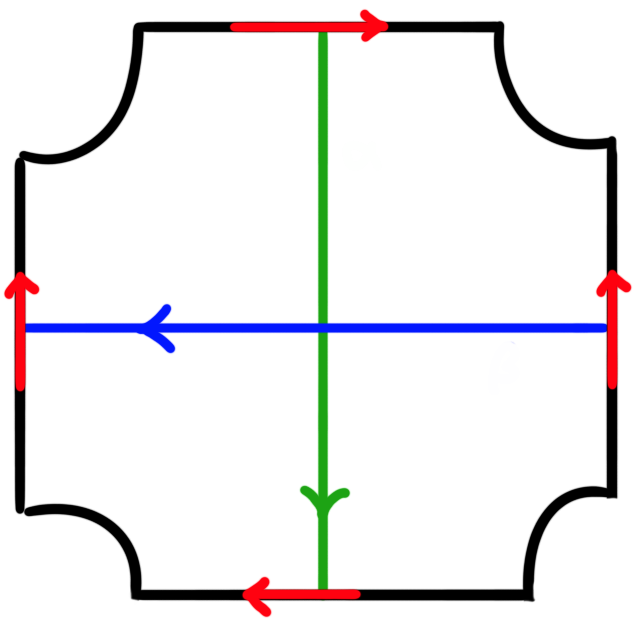}
\put(55,22){\large $\alpha$}
\put(32,50){\large $\beta$}
\end{overpic}
\caption{$\alpha,\beta$ in $\pklein$.}
\label{fig:pklein}
\end{subfigure}
\begin{subfigure}{.48\linewidth}
\centering
\begin{overpic}[scale=.18,percent]{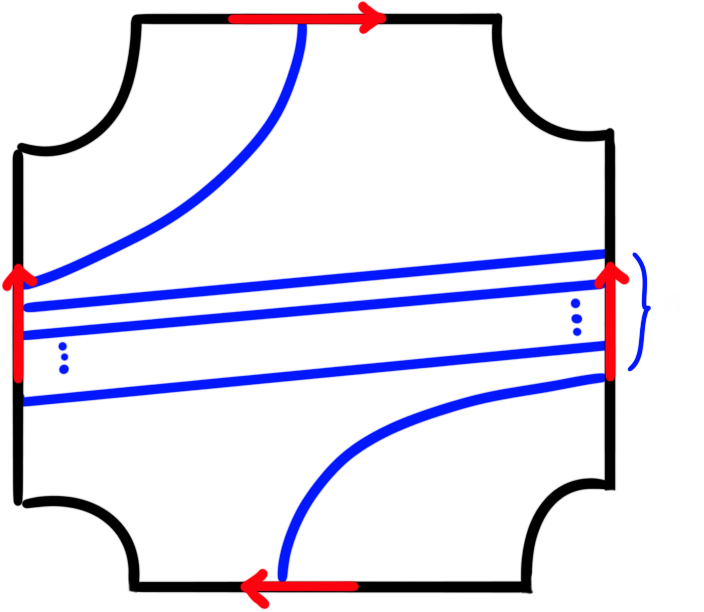}
\put(93,42){\large $n$}
\end{overpic}
\caption{$\gamma_n\simeq \alpha\beta^n$ in $\pklein$.}
\label{fig:pklein_eg}
\end{subfigure}
\caption{Non-characteristic annuli: type $\Kl$.}
\end{figure}  

\subsection{Type $\Mo$}\label{subsec:typeM}
Let $\pair$ be of type $\Mo$, and  
$X\subset \Compl\HK$ 
the I-fibered component. Denote by 
$\pi:X\rightarrow \pmobius$ the bundle projection
over a once-punctured M\"obius band $\pmobius$. 
Then up to isotopy, 
$\Compl\HK$ admits two non-characteristic annuli, since $X$ admits two essential, non-$\front$-parallel admissible annuli, which are given by the frontier of a regular neighborhood 
of the preimage of the circles $\alpha,\beta\subset\pmobius$ under $\pi$ in Fig.\ \ref{fig:pmobius}.

\subsection{Type $\Se$}\label{subsec:typeS}
Let $\pair$ be of type $\Se$, and 
$X\subset \Compl\HK$ the Seifert fibered component. Denote by $\pi:X\rightarrow \sdisk$ 
the fibration over
the disk $\sdisk$ with singularity $s\in\sdisk$ (see Fig.\ \ref{fig:sdisk}). 
Then $\Compl\HK$ admits two non-characteristic annuli, up to isotopy, given by 
the preimage of the circles $\alpha,\beta\subset \sdisk$ in Fig.\ \ref{fig:sdisk} under $\pi$.

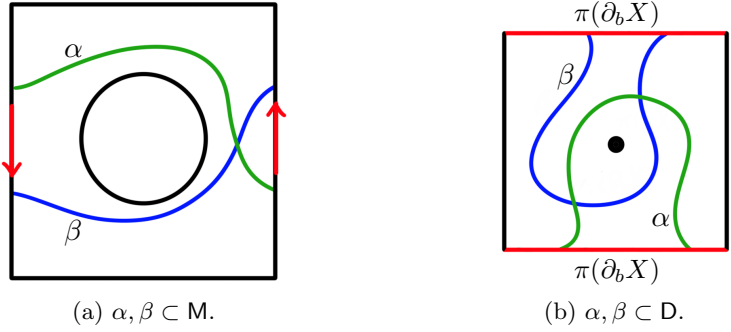
\begin{figure}[t]
\begin{subfigure}[t]{.48\linewidth}
\centering
\begin{overpic}[scale=.15,percent]{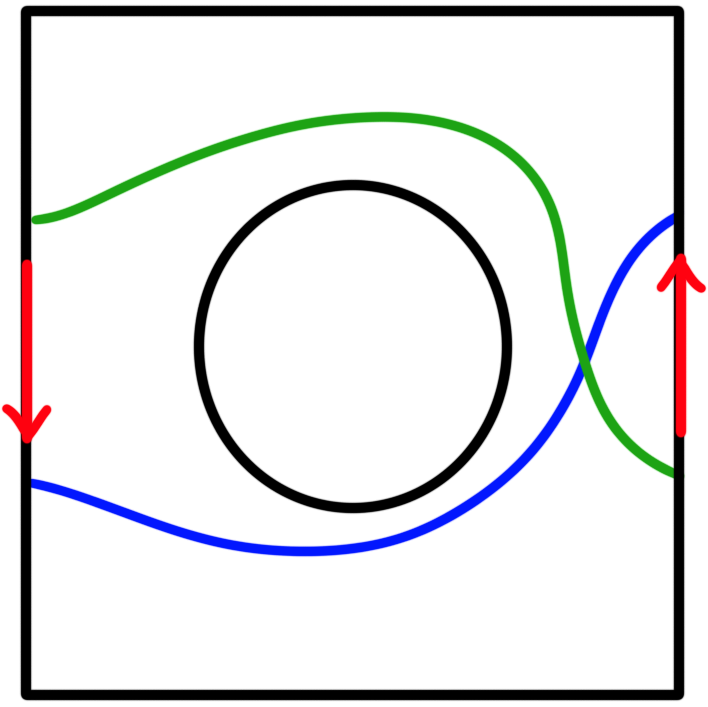}
\put(22,80){\large $\alpha$}
\put(22,15){\large $\beta$}
\end{overpic}
\caption{$\alpha,\beta\subset\pmobius$.}
\label{fig:pmobius}
\end{subfigure}
\begin{subfigure}[t]{.48\linewidth}
\centering
\begin{overpic}[scale=.15,percent]{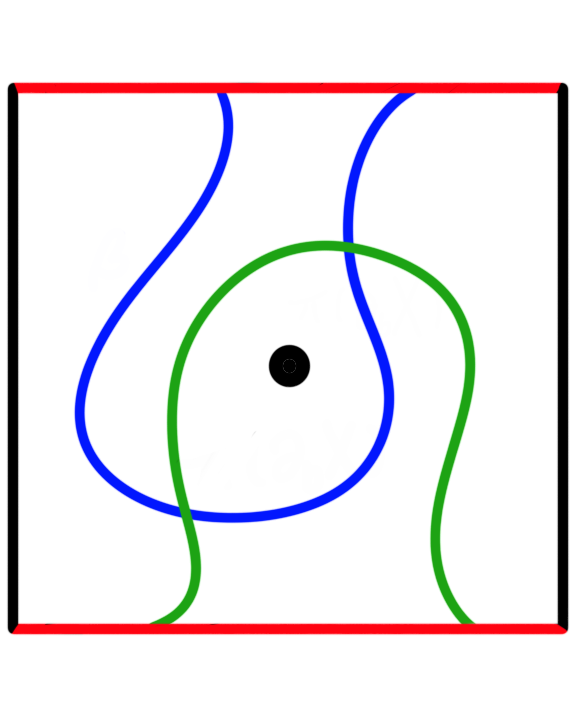}
\put(53,20){\large $\alpha$}
\put(20,70){\large $\beta$}
\put(26,1.5){$\pi(\partial_b X)$}
\put(26,92){$\pi(\partial_b X)$}
\end{overpic}  
\caption{$\alpha,\beta\subset\sdisk$.}
\label{fig:sdisk}
\end{subfigure}
\caption{Non-characteristic annuli: types $\Mo$ and $\Se$.}
\end{figure}



%


\section{Type $4$-$1$ annuli}\label{sec:fourone}

\subsection{Eudave-Mu\~noz knots}\label{subsec:emknots}
\begin{figure}[b]
\begin{subfigure}[t]{.47\linewidth}
\centering
\begin{overpic}[scale=.12,percent]{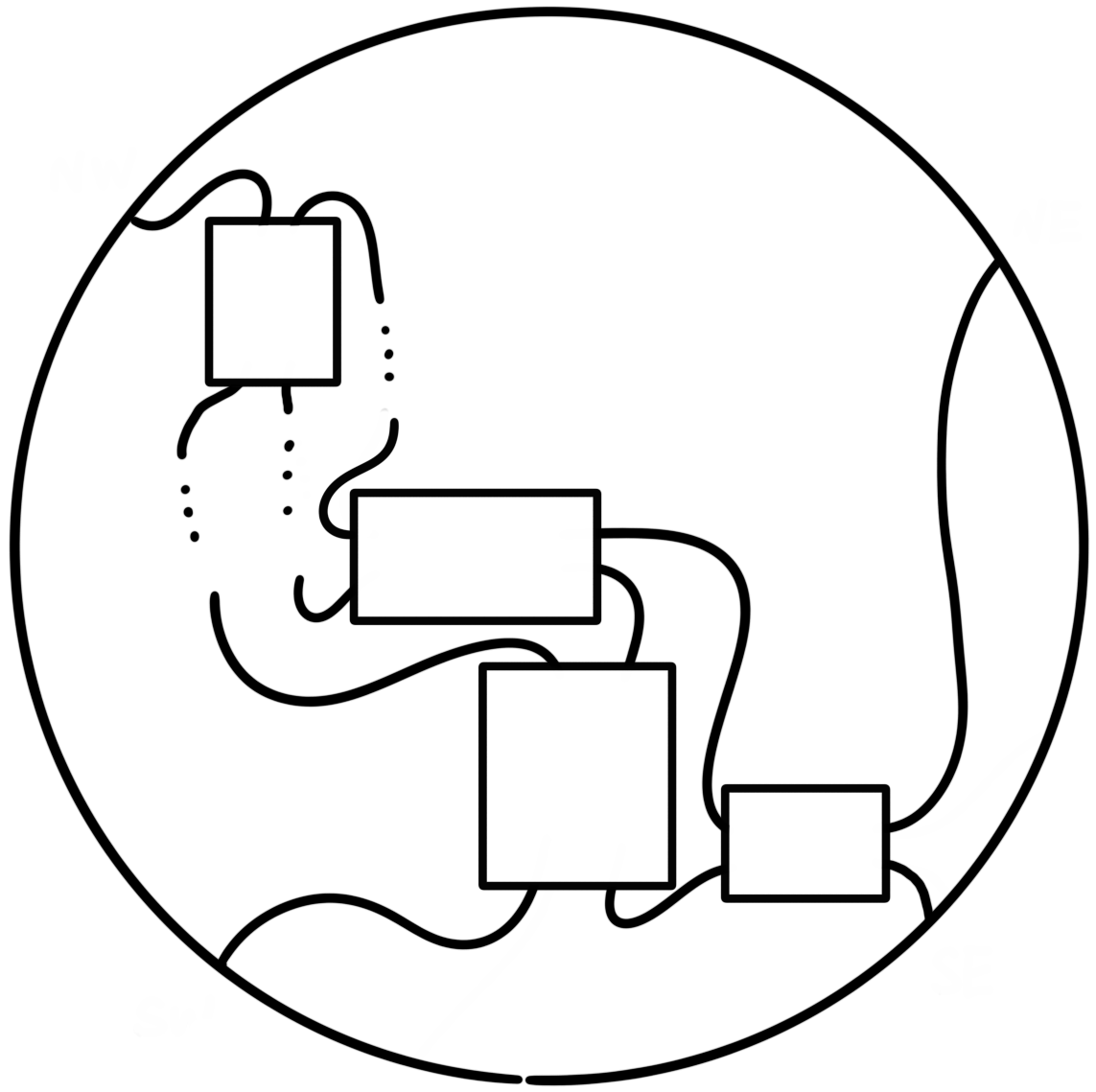}
\put(70,21){$a_n$}
\put(45,27){$a_{n-1}$}
\put(37,47){$a_{n-2}$}
\put(22,70){$a_1$}
\put(91,77){\footnotesize $\mathit{NE}$}
\put(85,11){\footnotesize $\mathit{SE}$}
\put(2,81){\footnotesize $\mathit{NW}$}
\put(10,5.3){\footnotesize $\mathit{SW}$}
\end{overpic}
\caption{$n$ is even.}
\label{fig:even}
\end{subfigure} 
\begin{subfigure}[t]{.47\linewidth}
\centering
\begin{overpic}[scale=.12,percent]{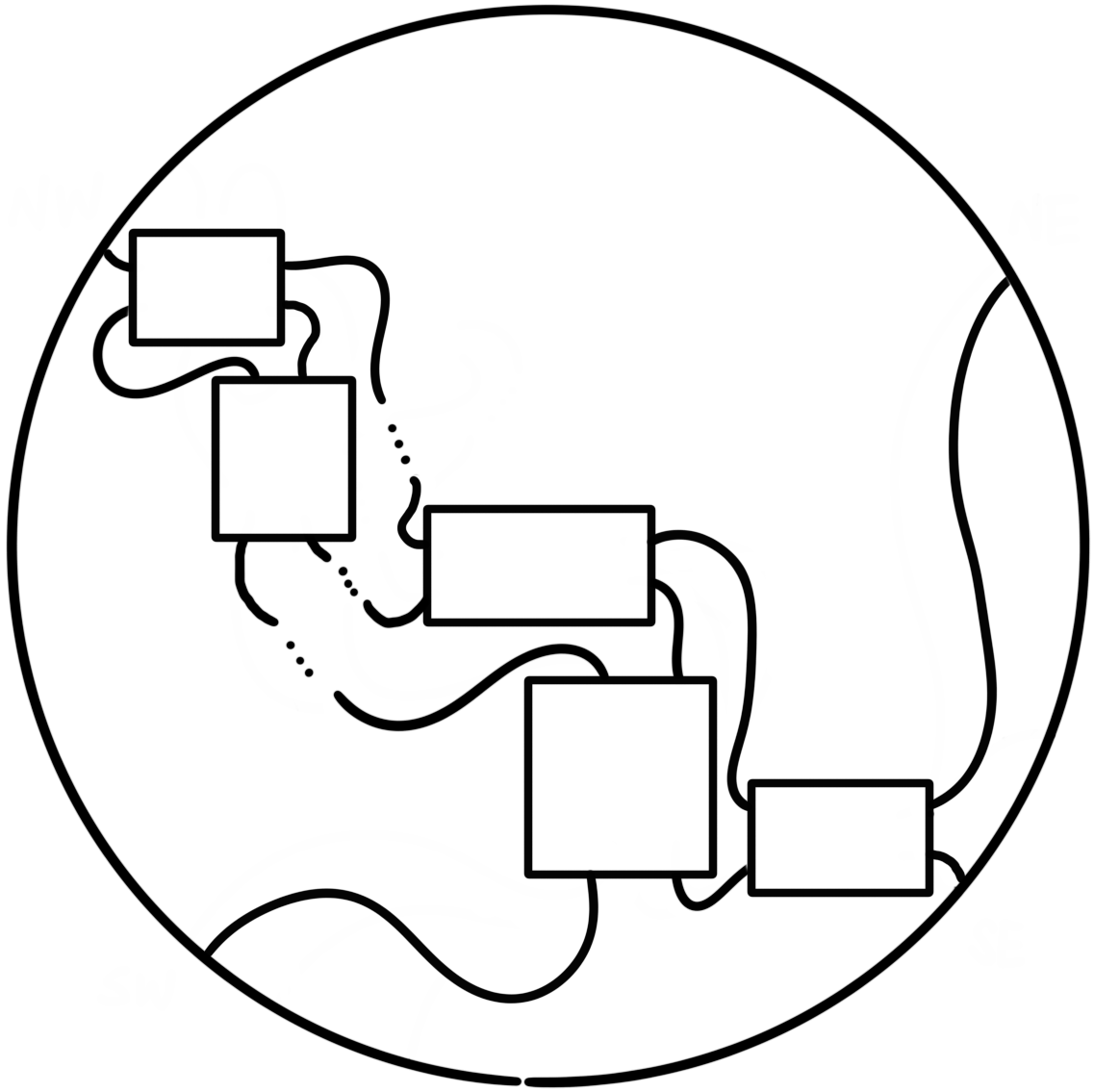}
\put(73,21.5){$a_n$}
\put(49.5,25){$a_{n-1}$}
\put(41.5,46){$a_{n-2}$}
\put(22,56){$a_2$}
\put(15,72){$a_1$}
\put(93,75){\footnotesize $\mathit{NE}$}
\put(87,13){\footnotesize $\mathit{SE}$}
\put(-0.8,78){\footnotesize $\mathit{NW}$}
\put(8.5,6.5){\footnotesize $\mathit{SW}$}
\end{overpic}
\caption{$n$ is odd.}
\label{fig:odd}
\end{subfigure}
\begin{subfigure}[t]{.38\linewidth}
\centering
\begin{overpic}[scale=.12,percent]{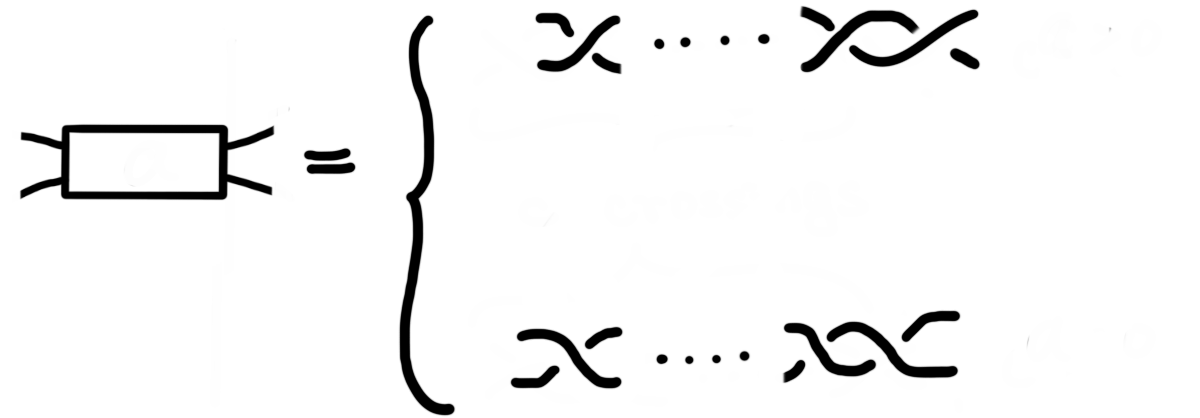}
\put(11,20){$a$}
\put(82,4){$a>0$}
\put(83,30){$a<0$}
\end{overpic}
\caption{Sign convention.}
\label{fig:horizontal}
\end{subfigure}
\hspace*{.04cm}
\begin{subfigure}[t]{.29\linewidth}
\centering
\begin{overpic}[scale=.12,percent]{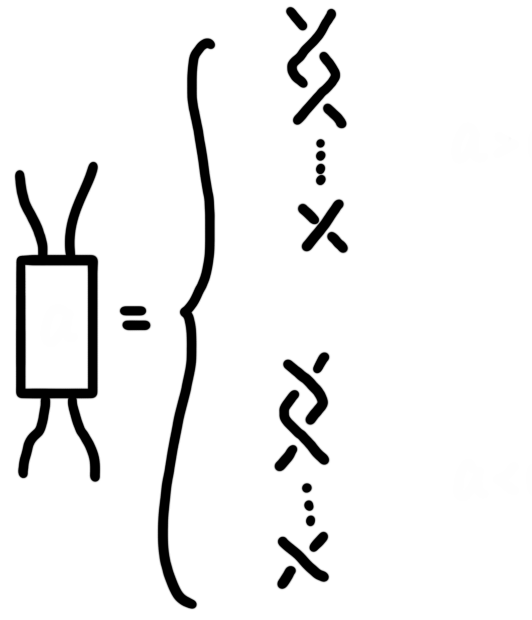}
\put(6,44){$a$}
\put(60,25){$a<0$}
\put(60,80){$a>0$}
\end{overpic}
\caption{Sign convention.}
\label{fig:vertical}
\end{subfigure}
\begin{subfigure}[t]{.29\linewidth}
\centering
\begin{overpic}[scale=.11,percent]{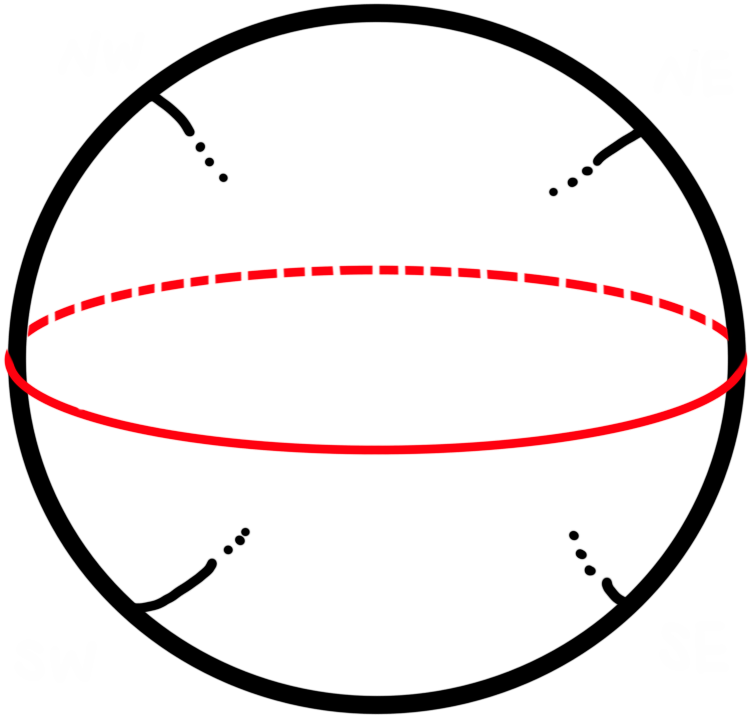}
\put(85,81){\footnotesize $\mathit{NE}$}
\put(85,7){\footnotesize $\mathit{SE}$}
\put(1.2,85){\footnotesize $\mathit{NW}$}
\put(1.2,2){\footnotesize $\mathit{SW}$}
\put(50,40){$e$}
\end{overpic}
\caption{Equator.}
\label{fig:equator}
\end{subfigure}
\begin{subfigure}[t]{.45\linewidth}
\centering
\begin{overpic}[scale=.11,percent]{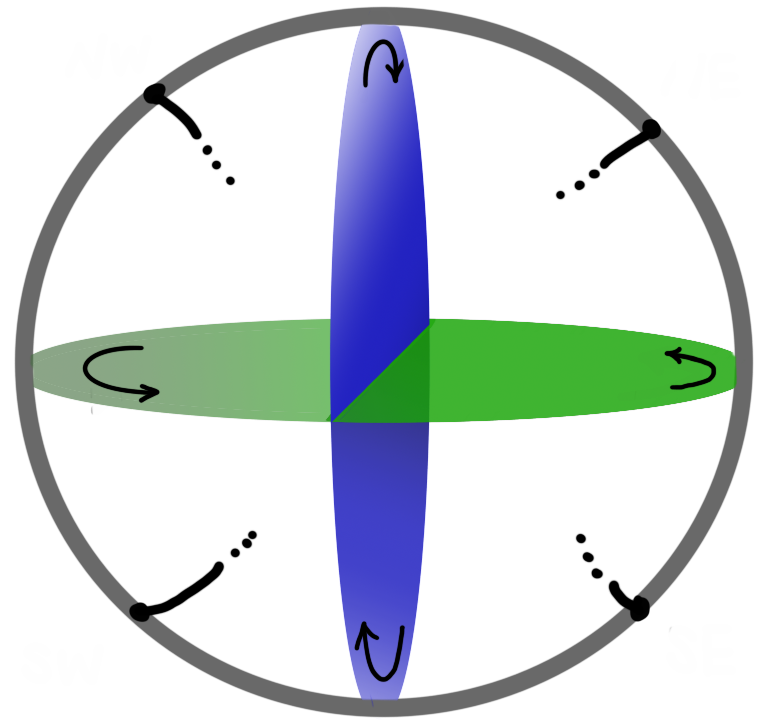}
\put(85,81){\footnotesize $\mathit{NE}$}
\put(85,7){\footnotesize $\mathit{SE}$}
\put(1.7,83){\footnotesize $\mathit{NW}$}
\put(2,5){\footnotesize $\mathit{SW}$}
\end{overpic}  
\caption{Twisting disks.}
\label{fig:twistingdisks}
\end{subfigure}
\begin{subfigure}[t]{.5\linewidth}
\centering
\begin{overpic}[scale=.13,percent]
{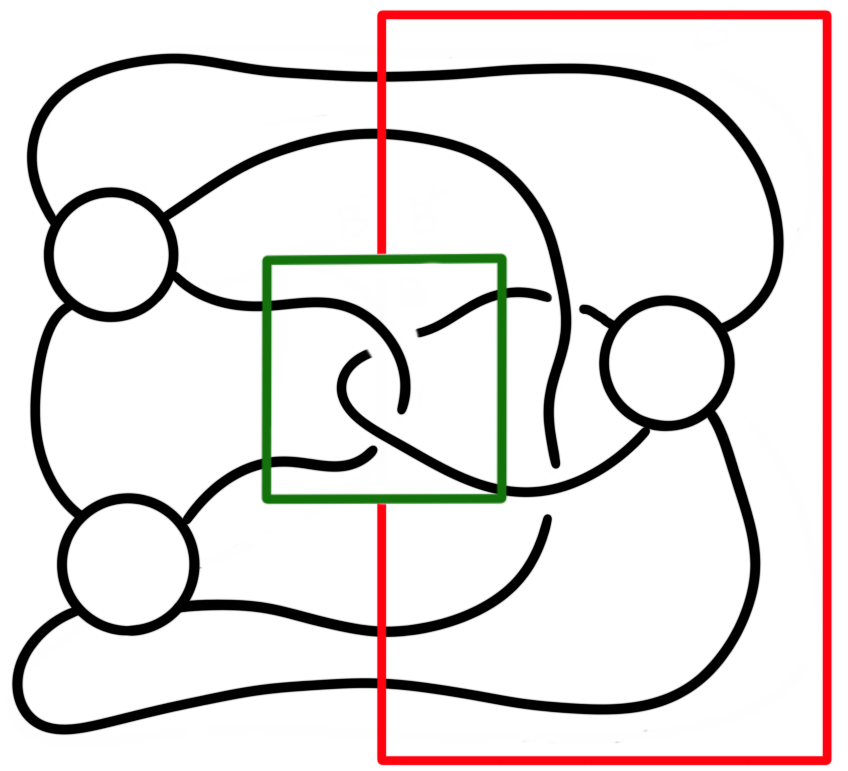}
\put(76,45){$\mathcal{C}$}
\put(9,58){$\mathcal{A}$}
\put(12,22){$\mathcal{B}$}
\put(48.5,40){$\ball'$}
\put(85.5,81){$\ball^\plus$}
\put(0.5,83){$\ball^\minus$}
\put(46,65){$\rmdisk$}
\put(70,10){$k_0$}
\end{overpic}
\caption{Eudave-Mu\~noz construction.}
\label{fig:emtangle}
\end{subfigure}
\caption{Rational and Eudave-Mu\~noz tangles.}
\end{figure}
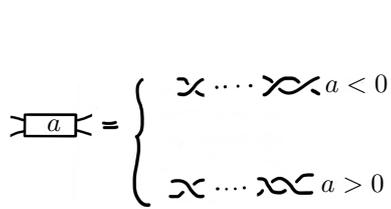
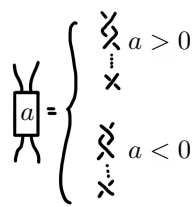
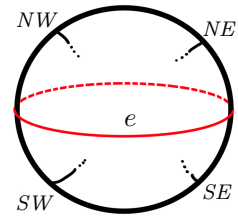
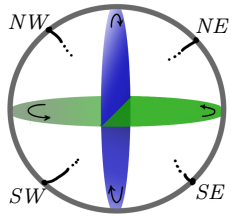
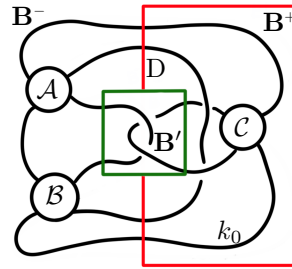
Consider the unit $3$-ball $\ball$, and set 
$\mathit{NE}:=\frac{1}{\sqrt{2}}(0,1,1)$, $\mathit{SE}:=\frac{1}{\sqrt{2}}(0,1,-1)$, 
$\mathit{NW}:=\frac{1}{\sqrt{2}}(0,-1,1)$, 
$\mathit{SW}:=\frac{1}{\sqrt{2}}(0,-1,-1)$.
Then a $2$-string tangle is a proper embedding of 
two arcs with the four endpoints 
$\mathit{NE},\mathit{SE},\mathit{NW}$, and $\mathit{SW}$. 
By convention, a diagram of a $2$-string tangle is the projection of the tangle onto the plane $x=0$.
   
A rational tangle 
$\Ra(a_1,\dots,a_n)$ is a $2$-string tangle 
in a $3$-ball $\ball$ given in Figs.\ \ref{fig:even}, \ref{fig:odd} 
with the sign convention in Figs.\ \ref{fig:horizontal}, \ref{fig:vertical}.
A rational tangle $\Ra(a_1,\dots,a_n)$ 
is uniquely determined, up to isotopy fixing the endpoints, by the rational number 
\begin{equation}\label{eq:rational}
[a_1,\dots,a_n]:=a_n+\frac{1}{a_{n-1}+\frac{1}{\dots+\frac{1}{a_1}}}.
\end{equation}

Consider the trivial knot $(\asphere, k_0)$ in Fig.\ \ref{fig:emtangle},
where $\mathcal{A}=\Ra(l)$, $\mathcal{B}=\Ra(p,-2,m,-l)$, 
and $\mathcal{C}=\Ra(-n,2,m-1,2,0)$ with 
$l,m,n,p\in\mathbb{Z}$.
Let $\ball'$ be the $3$-ball in the center of Fig.\ \ref{fig:emtangle}, which meets $k_0$ at two subarcs, and 
denote by $\pi:\sphere\rightarrow \asphere$ 
the double branched cover over $k_0$.
Then the Eudave-Mu\~noz knot $\pairem$
is defined to be the pair  
$\bigl(\sphere, c(\pi^{-1}(\ball'))\bigr)$ with $l,m,n,p$ not those 
forbidden integers in \cite[p.132]{Eud:02}.

The exterior of an Eudave-Mu\~noz knot
admits a canonical incompressible twice-punctured torus
$T$ given by the preimage of the disk $\rmdisk$ in Fig.\ \ref{fig:emtangle} under $\pi$. 
The twice-punctured torus cuts $\Compl \emK$ 
into two genus $2$ handlebodies since 
$\rmdisk$ cuts $\overline{\asphere-\ball'}$ 
into two $3$-balls $\ball^\pm$ so that $(\ball^\pm,k_0\cap \ball^\pm)$ are isotopic, without fixing the boundary, 
to a trivial $3$-string tangle. Therefore
$T$ induces two handlebody-knots $\pairpnind$ with $\HK^\pm_{(l,m,n,p)}:=\pi^{-1}(\ball^\pm)$, called the \emph{induced handlebody-knots} of $\pairem$; we drop the subscript of $\pairpnind$ when there is no risk of confusion.
 
Since $T\subset\Compl\emK$ is incompressible with non-integral boundary slope, 
the annuli $\pi^{-1}(\partial \ball' \cap \ball^\mp)$ are essential and of type $4$-$1$ in $\Compl {\HK^\pm}$, respectively. 
The essentiality of $T$ implies 
$\pairpn$ are non-trivial, and 
the hyperbolicity of $\pairem$ 
implies its atoroidality. 
By \cite{KodOzaGor:15} and \cite{GorLue:04}, the converse is also true.  
\begin{lemma}\label{lm:fourone_emknots}
If $\pair$ admits an essential annulus 
$A$ of type $4$-$1$, then 
$\pair$ is an induced handlebody-knot of 
some Eudave-Mu\~noz knot $\pairem$. 
\end{lemma}  
\begin{proof}
Let $V$ be the solid torus cut off by $A$ from $\Compl\HK$. 
Since $\partial\HK$ has no compressing disk in $\sphere$ disjoint from $A$, 
the twice-punctured torus $T:=\overline{\partial\HK-V}$ 
is incompressible in $\Compl V$. 
Furthermore, $T$ has a non-integral boundary slope with respect to $(\sphere,cV)$ by the essentiality of $A$.
Set $U:=\Compl{\HK\cup V}$, and observe that 
$U$ is $\partial$-reducible---otherwise, $\HK\cup V$ would be $\partial$-reducible, and by \cite[Lemma $3.9$]{KodOzaGor:15}, 
there would exist an essential disk in $\HK$ disjoint from $A$, contradicting $A$ being of type $4$-$1$. 
In particular, the frontier of the compression body of 
$U$ is two tori, one torus or the empty set. The atoroidality of $\pair$ excludes
the first two cases, so $U$ is a genus two handlebody. Applying \cite[Lemma $3.14$]{KodOzaGor:15}, we obtain $(\sphere,cV)$ is hyperbolic.  
Since $(\sphere,cV)$ admits a non-integral toroidal Dehn surgery, by Gordon-Luecke \cite{GorLue:04}, 
$(\sphere,cV)$ is an Eudave-Mu\~noz knot $\pairem$ with 
$T$ the canonical twice-punctured torus. 
\end{proof}

\subsection{Classification}\label{subsec:fourone_classification}
Consider the rational tangle $(\ball,t)=\Ra(a_1,\dots,a_n)$ with 
$\frac{p}{q}=[a_1,\dots,a_n]$, 
and $\pi_t:V\rightarrow \ball$ be the double-cover of $\ball$ branched along $t$. Denote by $e$ the equator $\partial\ball\cap \{z=0\}$, and 
let $l\subset \partial V$ be a component of $\pi^{-1}(e)$ (Fig.\ \ref{fig:equator}). 
Then we have the following
\begin{lemma}\label{lm:slope}
Let $g$ be a generator of $H_1(V)$. Then 
$[l]=\pm pg\in H_1(V)$. 
\end{lemma}
\begin{proof}
A disk $\rmdisk$ separating the two strings in $\Ra(a_1,\dots,a_n)$ 
can be constructed as follows: start with a disk separating
the two strings in $\Ra(0)$ (resp.\ $\Ra(0,0)$) if $n$ is odd (resp.\ even);
next, twist 
along $\ball\cap \{y=0\}$ (resp.\ $\ball\cap \{z=0\}$) $a_1$ times (Fig.\ \ref{fig:twistingdisks}); then inductively, for every $2\leq i\leq n$, 
twist along $\ball\cap \{z=0\}$ or $\ball\cap \{y=0\}$ $a_i$ times after twisting $a_{i-1}$ times along $\ball\cap \{y=0\}$ or $\ball\cap \{z=0\}$, respectively.

Note that the final twist twists the strings $a_n$ times along $\ball\cap \{y=0\}$, so the number of intersection between the separating disk $\rmdisk$ and the equator $e=\partial \ball\cap\{z=0\}$
is calculated by the numerator of \eqref{eq:rational}. Therefore $l$ meets a meridian disk of $V$ $p$ times, and thus the claim. 
\end{proof}

\begin{figure}[b]
\begin{subfigure}{.48\linewidth}
\centering
\begin{overpic}[scale=.17,percent]{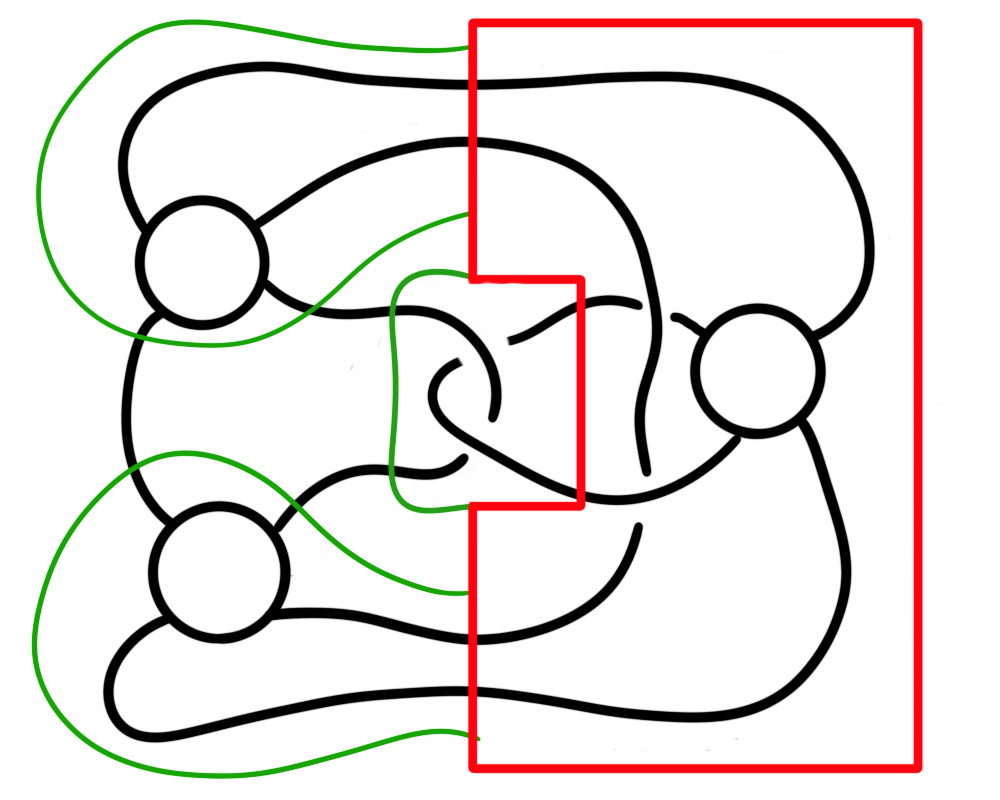}
\put(4,1.9){$\rmdisk_\beta$}
\put(2.6,75){$\rmdisk_\alpha$}
\put(0,40){\Large $\ball_y$}
\put(33,39){$\rmdisk'$}
\put(17.5,53){$\mathcal A$}
\put(20,21){$\mathcal B$}
\put(74,41){$\mathcal C$}
\put(70,11){$k_0$}
\put(70,60){\Large $\ball^\plus$}
\put(50.5,39){\large $\ball'$}
\end{overpic}
\caption{Essential annuli in $\Compl{\HK^+}$.}
\label{fig:hkplus}
\end{subfigure}
\begin{subfigure}{.48\linewidth}
\centering
\begin{overpic}[scale=.17,percent]{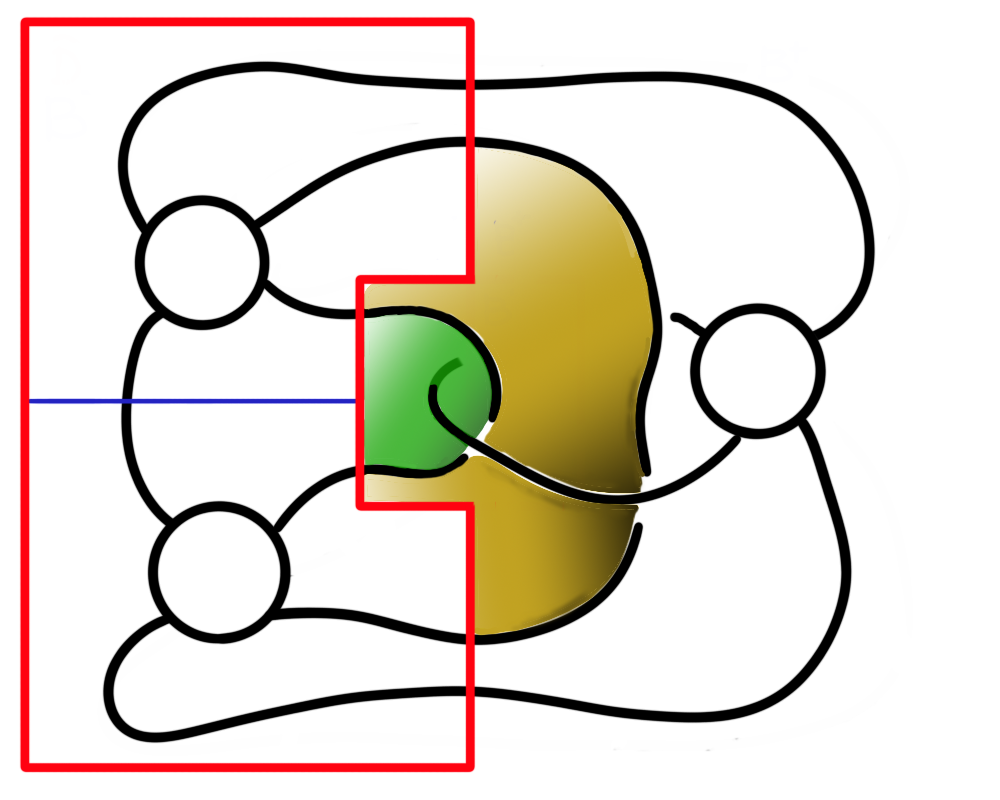} 
\put(17.5,53){$\mathcal A$}
\put(20,21){$\mathcal B$}
\put(74,41){$\mathcal C$}
\put(20,41.5){$\rmdisk_s$}
\put(50,57){$\rmdisk_a$}
\put(36.5,35){\footnotesize $\rmdisk_{em}$}
\put(4,70){\Large $\ball^\minus$}
\end{overpic}
\caption{Essential annuli in $\Compl{\HK^-}$.}
\label{fig:hkminus}
\end{subfigure}
\caption{Essential annuli in $\Compl{\HK^\plusminus}$.}
\end{figure}

\begin{theorem}\label{teo:fourone_classification}
Type $4$-$1$ annuli are always non-characteristic, and if $\Compl\HK$ admits a type $4$-$1$ annulus, then
$\charE$ is either \includegraphics[scale=.17]{100i}
or \includegraphics[scale=.17]{200i}.
Furthermore, it is the latter if and only if 
\[\pair\simeq \pairpind{l,m,n,p}\]
for some $l,m,n,p$ with $l\neq \pm 2$ 
and $2mpl-2p-pl-ml+1\neq \pm 2$.
\end{theorem}
\begin{proof}
By Lemma \ref{lm:fourone_emknots}, 
$\pair$ is equivalent to an induced 
handlebody-knot 
of an Eudave-Mu\~noz knot $\pairem$,
for some $(l,m,n,p)$.

Suppose it is equivalent to 
$\pairp$. Then observe that the disks
$\rmdisk_\alpha, \rmdisk_\beta,\rmdisk'$ in Fig.\ \ref{fig:hkplus} cut 
off the rational tangles 
$\mathcal{A},\mathcal{B},(\ball',\ball'\cap k_0)=\Ra(-2,0)$ from the exterior of $\ball^+$, respectively. 
Denote by $X_\alpha, X_\beta, X'$ the solid tori cut off by the annuli
$A_\alpha:=\pi^{-1}(\rmdisk_\alpha),A_\beta:=\pi^{-1}(\rmdisk_\beta),
A':=\pi^{-1}(\rmdisk')$ from $\Compl{\HK^+}$. 
Consider now the preimage $Y$ of 
the ball $\ball_y$ bounded by $\rmdisk_\alpha, \rmdisk_\beta,\rmdisk'$, 
and observe that there exists an I-bundle structure $\pi:Y\rightarrow P$ over a pair of pants $P$ with components of $\pi^{-1}(\partial P)$ 
corresponding to $A_\alpha,A_\beta$ and $A'$.
Since the cokernel of the induced homomorphism $H_1(A')\rightarrow H_1(X')$ by the inclusion 
is $\mathbb{Z}_2$, the I-bundle structure can be extended to the union $Y\cup X'$ so that $Y\cup X'$ is I-fibered over a once-punctured M\"obius band; $A$ 
is hence non-characteristic.

On the other hand, by Lemma \ref{lm:slope}, 
the cokernel of the induced homomorphism $H_1(A_\ast)\rightarrow H_1(X_\ast)$ by inclusion 
is a cyclic group of order $\mathfrak{o}_\alpha:=\vert l\vert$ 
(resp.\ order $\mathfrak{o}_\beta:=\vert 2lmp-lp-lm-2p+1\vert$)
if $\ast=\alpha$ (resp.\ $\ast=\beta$), so at most one of $\mathfrak{o}_\alpha,\mathfrak{o}_\beta$ 
is $\pm 2$. If $\mathfrak{o}_\alpha$ (resp.\ $\mathfrak{o}_\beta$) is $\pm 2$, 
then the I-bundle structure of $Y\cup X'$ 
can be further extended so that 
$Y\cup X'\cup X_\alpha$ (resp.\ $Y\cup X'\cup X_\beta$) 
is I-fibered 
over a once-punctured Klein bottle,
and $\charE$ 
is hence \includegraphics[scale=.15]{100i}.  
If none of $\mathfrak{o}_\alpha,\mathfrak{o}_\beta$ is $\pm 2$,
then the I-bundle cannot be extended, so $\charE$ is \includegraphics[scale=.15]{200i}; this gives us the constraints on $l,m,n,p$ in the second assertion.

Suppose $\pair$ is equivalent to
$\pairn$. Then we observe that the preimage $M$ of the 
disk $\rmdisk_{em}$ in Fig.\ \ref{fig:hkminus} is an 
essential M\"obius band, 
whose core is $K_{(l,m,n,p)}$. 
Observe also that the preimage $A$ of the 
disk $\rmdisk_a$ in Fig.\ \ref{fig:hkminus} 
is a type $3$-$3$ essential annulus since 
the preimage of the disk $\rmdisk_s$ is an essential disk in $\HK^-$ separating components of $\partial A$.  
In addition, $A, M$ meet at the arc $\pi^{-1}(\rmdisk_{em}\cap \rmdisk_a)$, and thus a regular neighborhood of $A\cap M$ can be I-fibered over a once-punctured Klein bottle. This implies $\charE$
is \includegraphics[scale=.15]{100i}, and $A$ is non-characteristic. 
\end{proof}

\subsection{On Theorem \ref{intro:teo:ann_diag}}\label{subsec:classification}
We explain here how the table in Fig.\ \ref{tab:ann_daigram} is derived. 
%
%
Recall first that, given a type $3$-$3$
annulus $A\subset\Compl\HK$, there exists a unique separating disk $\Disk A\subset \HK$ disjoint from $\partial A$ \cite[Lemma $2.3$]{FunKod:20}. The disk 
$\Disk A$ cuts 
$\HK$ into two solid tori $V_1,V_2$, each of which contains a component of $\partial A$. 
The \emph{slope pair} of $A$ is then defined to be 
the unordered pair $(r_1,r_2)$ with $r_1,r_2$ the slopes 
of components
of $\partial A$ with respect to $V_1,V_2$ in $\sphere$. 
It is known that the slope pair has either of the form
$(\frac{p}{q},\frac{q}{p})$, $pq\neq 0$, or of the form $(\frac{p}{q},pq)$, $q>0, p\neq \pm 1$ \cite[Lemma $2.12$]{Wan:23}. In the case $(r_1,r_2)=(0,0)$, 
we say $A$ has a trivial 
slope. 

\begin{lemma}\label{lm:ltwo_to_trivial_slope}
If $A\subset\Compl\HK$ is of type $3$-$3$ii, then $A$ has a trivial slope. 
\end{lemma}
\begin{proof}
Let $V_1,V_2\subset\HK$ be the solid tori cut off by the disk disjoint from $A$, and $l_1,l_2$ be the components of $\partial A$ in $V_1,V_2$, respectively. 
The annulus $A$ is compressible
since $A\subset \Compl{V_1\cup V_2}$ is inessential and $A$ meets both $V_1,V_2$. 
Any compressing disk of $A$ induces two disks $D_1,D_2\subset\Compl{V_1\cup V_2}$ with 
$\partial D_i=l_i$, and therefore the claim.
\end{proof}

Conversely, a type $3$-$3$ annulus with trivial slope 
may not be of type $3$-$3$ii, but we have the following.
\begin{lemma}\label{lm:trivial_slope_to_ltwo}
Suppose $A\subset\Compl\HK$ is of type $3$-$3$ with trivial slope, 
and there exists a type 
$2$-$2$ annulus $A'\subset\Compl\HK$ disjoint from $A$. 
Then $A$ is of type $3$-$3$ii. 
\end{lemma}
\begin{proof}
Let $l_1,l_2$ be the components of $\partial A$, and $l,l_s$ the components of $\partial A'$ 
with $l_s$ bounding an essential separating disk $D_s$. Since $\partial D_s\cap \partial A=\emptyset$, the disk $D_s$ cuts $\HK$ into 
two solid tori $V_1,V_2$ with $l_i\subset V_i$, $i=1,2$.   
This implies $l$ is parallel to either $l_1$ or $l_2$, say $l_1$; the union $A'\cup D_s$ then induces a disk in 
$\Compl{V_1\cup V_2}$ bounded by $l_1$, 
and hence $A\subset \Compl{V_1\cup V_2}$ is inessential. 
\end{proof}

Recall the classification of relative JSJ-graphs of handlebody-knots whose exteriors contain 
a type $2$ annulus.  
\begin{lemma}[{\cite[Theorems $1.4$ and $1.6$]{Wan:22p}}]\label{lm:typetwo}\hfill
\begin{enumerate}[label=\textnormal{(\roman*)}]
\item\label{itm:twoone} If $\Compl\HK$ contains a type $2$-$1$ annulus, 
then $\annhk$ is 
\raisebox{-.3\height}{\includegraphics[scale=.25]{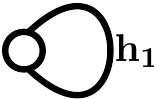}}  
\item\label{itm:twotwo}  
If $\Compl\HK$ contains a type $2$-$2$ annulus, 
then $\annhk$ is one of the following:
\[
\includegraphics[scale=.23]{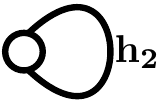},\quad
\includegraphics[scale=.23]{r210h.hopf},\quad
\includegraphics[scale=.28]{r303hi}
.\]
\end{enumerate}  
\end{lemma}
 
For type $3$-$3$ii annuli, \cite[Theorem $1.6$]{Wan:22p} and Lemma \ref{lm:ltwo_to_trivial_slope} implies the following.
\begin{lemma}\label{lm:typethreethree_ii} 
If $\Compl\HK$ contains a type $3$-$3$ii 
annulus, then $\annhk$ is one of the following:
\[
\includegraphics[scale=.25]{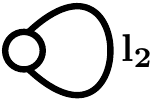},\quad
\includegraphics[scale=.27]{r303hi}
.\]
\end{lemma}

In the case $\charE$ is a $\theta$-graph, we have 
the converse of Lemma \ref{lm:typetwo}\ref{itm:twotwo}
by Lemma \ref{lm:trivial_slope_to_ltwo} and Theorem \cite[Theorem $1.5$]{Wan:22p}.  
\begin{lemma}\label{lm:theta}
If $\charE$ is \raisebox{-.3\height}{\includegraphics[scale=.25]{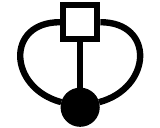}}, 
then $\annhk$ is 
\raisebox{-.3\height}{\includegraphics[scale=.25]{r303hi}}.
\end{lemma}

\begin{lemma}[{\cite[Lemma $2.3$]{Wan:23p}}]\label{lm:i_stick}
If $\charE$ is \includegraphics[scale=.2]{100i}, then $\annhk$ is \includegraphics[scale=.22]{r100i}. 
\end{lemma}

\begin{remark}
The notation is different in \cite{Wan:23p}. 
First, the JSJ-graph is called characteristic diagram there, and it does not distinguish I-fibered and Seifert fibered components---both are filled circles.
Secondly, for a type $3$-$2$ 
annulus $A$, there is a well-defined slope $r$, that is, the slope of 
the core of $A$ with respect to the solid torus cut off by $A$ from $\Compl\HK$, and hence the label $\knot_1(r)$ in 
\cite[Lemma $2.3$]{Wan:23p}.
\end{remark}

\subsubsection*{Proof of Theorem \ref{intro:teo:ann_diag}} 
Note first that Figs.\ \ref{fig:a100h}, \ref{fig:a200hi},
and \ref{fig:a300hi} follow directly from Theorem \ref{teo:fourone_classification} since 
no type $4$-$1$ annulus is characteristic.
Fig.\ \ref{fig:a100i} is Lemma \ref{lm:i_stick}, and there is nothing to prove about Figs.\ \ref{fig:a110h_hopf}, \ref{fig:a110h_link} 
since all cases may occur.
To see Figs.\ \ref{fig:a210h.hopf} and \ref{fig:a210h.link}, we note that the 
loop edge cannot represent a type $2$-$1$ 
or a type $3$-$3$ii annulus by
Lemmas \ref{lm:typetwo}\ref{itm:twoone}
and \ref{lm:typethreethree_ii}. 
Fig.\ \ref{fig:a201h} is a consequence of Lemmas \ref{lm:typetwo} and \ref{lm:typethreethree_ii}, 
since no other types of annuli can occur as an 
edge of a bigon, and Fig.\ \ref{fig:a303hi} is a result of
Lemma \ref{lm:theta}.

Lastly, consider Fig.\ \ref{fig:a301hi}.
Let $A,A',A''$ be the annuli corresponding to 
the two edges of the bigon and the other edge, 
respectively. 
As with the previous case, $A,A'$ are of 
type $3$-$3$i; in addition, by \cite[Corollaries $3.5$, $3.10$]{Wan:22p}, 
$A,A'$ have the same slope pair $(\frac{p}{q},pq)$ 
with $\vert p\vert>1$ and $\partial A,\partial A'$ are parallel. On the other hand, by Theorem \ref{teo:fourone_classification}, 
$A''$ is of type $3$-$2$, so it suffices to show $A''$ 
cannot be of type $3$-$2$ii. 

Let $X$, $Y$ be the solid tori cut off from $\Compl\HK$ by 
$A\cup A'$, $A''$, respectively, and $Z$
the exterior of $X\cup Y$ in $\Compl\HK$. 
Also, let $l_1,l_2$ (resp.\ 
$l_1',l_2'$) be the components of $\partial A$ (resp.\ $\partial A'$); it may be 
assumed that $l_i,l_i'$ are parallel in $\partial\HK$, 
$i=1,2$, and there exists a unique essential separating disk 
disjoint from $l_i,l_i'$, $i=1,2$.
Let $\rnbhd{l_i}$ be a regular neighborhood of $l_i$ disjoint from $\partial A',\partial A''$, and  
$P$ be the $4$-punctured sphere $\partial\HK-\openrnbhd{l_1} \cup\openrnbhd{l_2}$.

Note that $l_1',l_2'$ cut off two annuli $B_1,B_2$ 
from $P$, respectively, and $\partial A''$ cuts off an annulus 
$B''$ from $P$ with $B''$ disjoint from $B_1\cup B_2$.
Let $l''$ be a component of $\partial B''$. Then
$l''$ either is parallel to a component of $\partial P$
or cuts $P$ into two pairs of pants. 

Suppose $l''$ is parallel to a component of $\partial P$. Then it may be assumed that 
$\partial B''$ and $\partial B_2$ are parallel, 
and hence $B''\cup B_2$ cuts an annulus $B$ 
from $P$. The frontier of a regular neighborhood of 
$A'\cup B\cup A''$ consists of three components, and the one in $Z$ is an annulus $\hat A$ of type $3$-$3$ 
since $\partial \hat A$ is parallel to $\partial A$
and also to $\partial A'$. In particular, $\hat A$ is essential and admissible in $Z$; also, it is non-$\front$-parallel since $\gb Z=2$, contradicting $Z$ is hyperbolic 
or I-fibered over a pair of pants.
 
Suppose $l''$ cuts $P$ into two pairs of pants. Let $l_{i\plusminus}$
be the components of $\partial P$ 
that meet $\rnbhd{l_i}$, $i=1,2$.
Then $l''$ either separates 
$l_{1\plus}\cup l_{1\minus}$ from $l_{2\plus}\cup l_{2\minus}$ 
or separates $l_{1\plus}\cup l_{2\plusminus}$ 
from $l_{1\minus}\cup l_{2\minusplus}$. 
In the former, $[l'']=0$ in $H_1(\Compl\HK)$, 
contradicting $A''$ is 
essential and separating.
In the latter, 
we have $[l'']=[l_1]\pm [l_2]$ in $H_1(Z)$. Consider the exterior $\Compl Y=X\cup Z$ of $Y\subset \Compl\HK$, and note that by \cite[Lemma $2.1$]{Wan:23p}, to see $A''$ cannot be of type $3$-$2$ii, it suffices to show that 
the quotient group
$H_1(X\cup Z)/\langle [l'']\rangle$
is not $\mathbb{Z}$. 
Since $[l_1]=[l_2]=px\in H_1(X\cup Z)$, for some $x\in H_1(X\cup Z)$, 
$[l'']$ is either $2px$ or $0$ in $H_1(X\cup Z)$; in neither case, $H_1(X\cup Z)/\langle [l'']\rangle$ 
is $\mathbb{Z}$.

 

\section{Non-characteristic annuli}\label{sec:nonchar}
We investigate how different types of non-characteristic annuli may occur in $\Compl\HK$.


 
\subsection{Types $\Mo$ and $\Se$}
Suppose $\pair$ is of type $\Mo$ or of type $\Se$.  
Then by Sections \ref{subsec:typeM}-\ref{subsec:typeS}, $\Compl\HK$ admits two non-characteristic annuli. To examine 
possible types of the two non-characteristic annuli, we consider first the following criterion for a type $3$-$2$ annulus to be of type $3$-$2$ii. 
Recall that given a type $3$-$2$ annulus $A$, there exists a unique non-separating disk 
$\Disk A\subset \HK$ disjoint from $\partial A$ by 
\cite[Lemma $3.6$]{FunKod:20}. Denote by $V\subset \HK$ the solid torus cut off by $\Disk A$.  
 
\begin{lemma}\label{lm:cri_threetwo_ii}
Given a type $3$-$2$ annulus $A$, 
it is of type $3$-$2$ii if and only if 
$(\sphere,cV)$ is a trivial knot.
\end{lemma}
\begin{proof}
The ``if'' direction is clear as no 
essential annulus exists in a solid torus. 

To see the ``only if'' direction, we let 
$X$ be the solid torus cut off by $A$ from $\Compl\HK$, 
and $B_x$ the annulus $X\cap V$. 
Denote by $Y$ the exterior of $X$ in $\Compl{V}$
and by $B_y$ the annulus $Y\cap V$. 
Note also that $\rnbhd{\Disk A}\subset Y$.

Suppose now that $A$ is of type $3$-$2$ii. 
Then $A$ is compressible or $\partial$-compressible in $\Compl V$. 
If $A$ is compressible, and $D$ is a compressing disk, then either $D\subset X$
or $D\subset Y$. The former is impossible
since $A\subset\Compl\HK$ is incompressible; the latter cannot happen either, for if it does, then 
$D$ is a disk in $\Compl X$, 
and hence the core of $A$ is a longitude of 
$X\subset \sphere$, and $A$ is then parallel to $B_x$ through $X$.  
Therefore $A$ is $\partial$-compressible in $\Compl V$. 

Suppose $D$ is a $\partial$-compressing disk of $A$ 
in $\Compl V$. Then $D\subset Y$, and 
hence $A$ is parallel to $B_y$ through $Y$ in $\Compl V$. In particular, $V\cup Y$ 
is a solid torus, and $(\sphere, cV)$ and 
$(\sphere, c(V\cup Y))$ are equivalent. 
Since $X=\Compl{V\cup Y}$ is a solid torus, so $(\sphere,cV)$ is trivial.
\end{proof}

Suppose now $\pair$ is of type $\Mo$, and suppose, 
in addition $\Compl\HK$ 
admits a type $4$-$1$ annulus $A_1$;
note that by Theorem \ref{teo:fourone_classification},
$A_1$ is non-characteristic, and there is an Eudave-Mu\~noz knot $\pairem$ 
whose induced handlebody-knot $\pairp$ is 
equivalent to $\pair$. Then we have the following. 

\begin{theorem}\label{teo:typeM}
The other non-characteristic annulus $A_2\subset\Compl\HK$ is of type $3$-$2$, and it is of type $3$-$2$ii if and only if $p=0,-1$.
\end{theorem}
\begin{proof}
It may be assumed that $A_1$ 
is $\pi^{-1}(\rmdisk')$ in Fig.\ \ref{fig:hkplus}, where 
$\pi:\sphere\rightarrow \asphere$ 
is the double-cover branched along $k_0$. 
Observe then that $\ball^+,k_0$ in Fig.\ \ref{fig:hkplus} 
can be deformed into Fig.\ \ref{fig:typeM_nonchar}. 
The other non-characteristic annulus $A_2$ is then given by $\pi^{-1}(\rmdisk'')$. Since the preimage $\Disk A$ 
of $\rmdisk_n\subset \ball^+$ in Fig.\ \ref{fig:typeM_nonchar} 
is a non-separating disk disjoint from $\partial A_2$; 
thus $A_2$ is of type $3$-$2$.

To see the second claim, we note first 
that $\Disk A$ cuts a solid torus 
$V$ off from $\HK$. Set 
$\ball_v:=\pi(V)$ and 
$\ball_x:=\pi(\Compl V)$
(see Fig.\ \ref{fig:typeM_BvBx}).
Secondly, the disk $\rmdisk_p$ in $\ball_x$ (see Fig.\ \ref{fig:typeM_Dp})
implies that the tangle 
$\Ta=(\ball_x,\ball_x\cap k_0)$ 
is a sum of the rational tangles 
$\Ra(2,0)$ and $\Ra(-p,2,0)$, and therefore $\Ta$ is trivial 
if and only if $R(-p,2,0)$ is integral 
if and only if $p=0$ or $-1$. 
On the other hand, by Lemma \ref{lm:cri_threetwo_ii}, $A_2$ is of type $3$-$2$ii 
if and only if $\Compl V$ is a solid torus
if and only if $\Ta$ is trivial, and the claim thus follows.
\end{proof}

\begin{figure}[t]
\begin{subfigure}{.48\linewidth}
\centering
\begin{overpic}[scale=.16, percent]{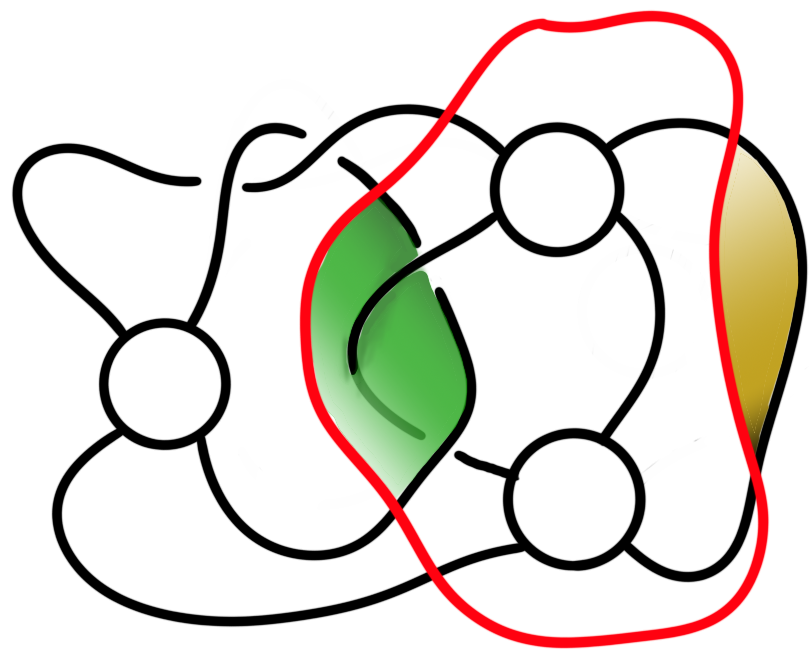}
\put(89,46){$\rmdisk_n$}
\put(45.7,33){$\rmdisk''$}
\put(66,55){\large $\mathcal A$}
\put(68,16){\large $\mathcal B$}
\put(18,30){\large \reflectbox{$\mathcal C$}}
\put(30,73){\Large $\ball^\plus$}
\end{overpic}
\caption{Disks $\rmdisk''$ and $\rmdisk_n$.}
\label{fig:typeM_nonchar}
\end{subfigure}
\begin{subfigure}{.48\linewidth}
\centering
\begin{overpic}[scale=.16, percent]{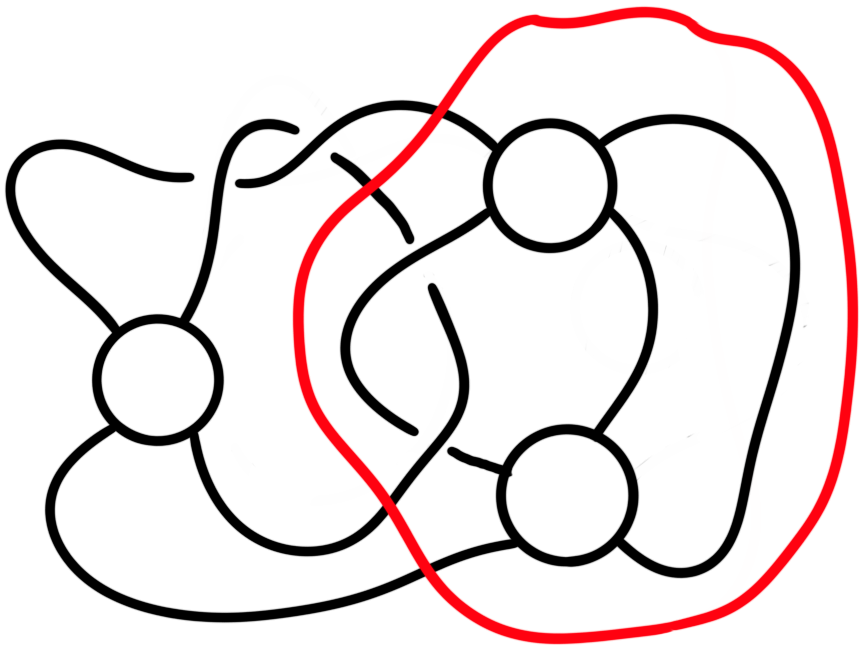}
\put(61,51){\large $\mathcal A$}
\put(63.5,14.5){\large $\mathcal B$}
\put(16.5,28){\large \reflectbox{$\mathcal C$}}
\put(10,68){\Large $\ball_v$}
\put(70,65){\Large $\ball_x$}
\end{overpic}
\caption{$V$ and $\Compl V$.}
\label{fig:typeM_BvBx}
\end{subfigure} 
\hfill
\vspace*{.2cm}
\begin{subfigure}{.32\linewidth}
\centering
\begin{overpic}[scale=.12, percent]{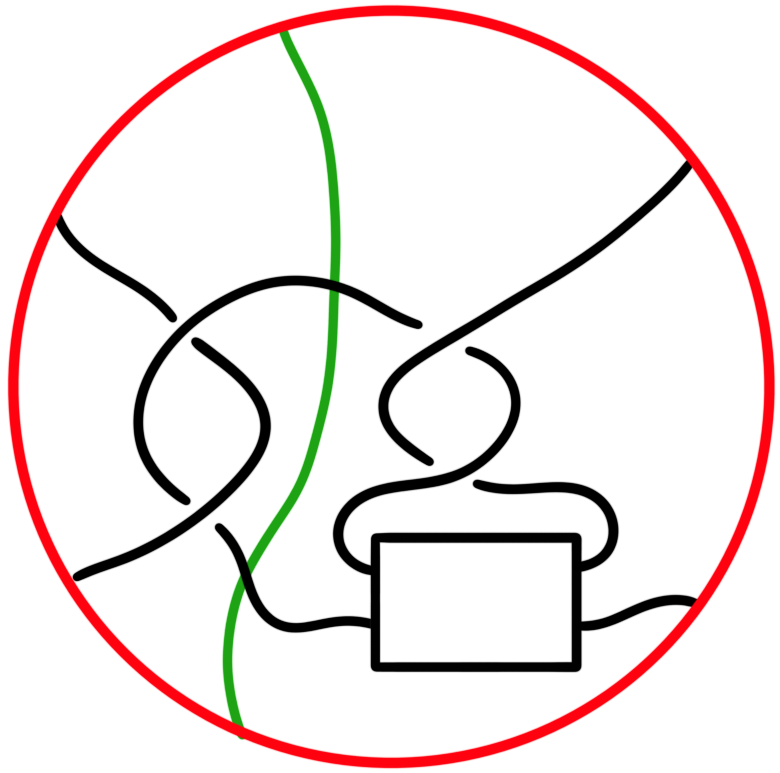}
\put(44,74){$\rmdisk_p$}
\put(58,20){\large $p$}
\put(85,90){\Large $\ball_v$}
\end{overpic}
\caption{$\Compl V$.}
\label{fig:typeM_Dp}
\end{subfigure}
\begin{subfigure}{.32\linewidth}
\centering
\begin{overpic}[scale=.12, percent]{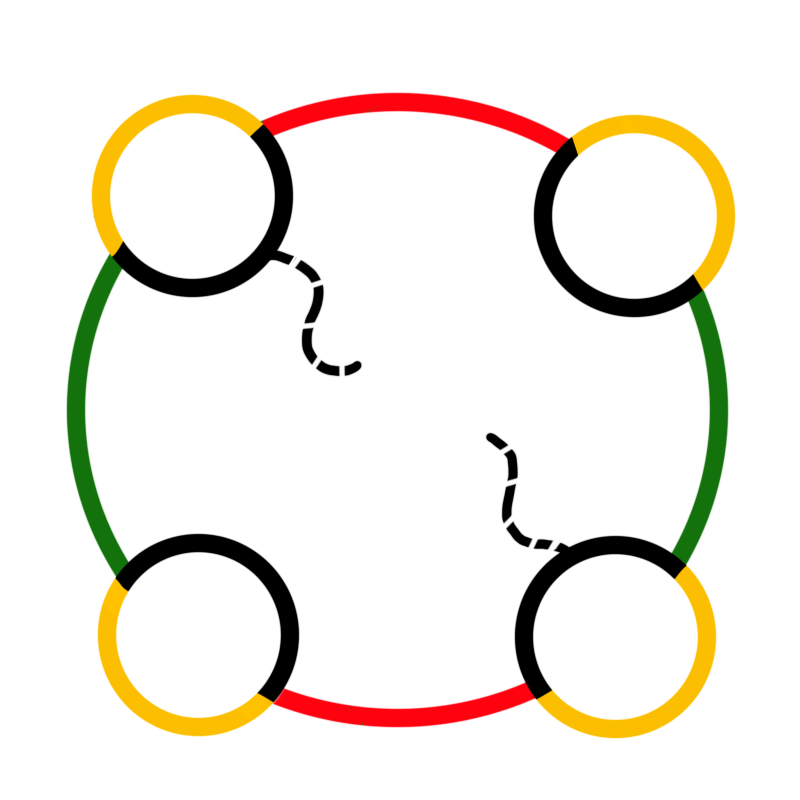}
\put(75,69){\large $V'$}
\put(73,15){\large $V$}
\put(19,16){\large $V'$}
\put(20,71){\large $V$}
\put(50,90){$A$}
\put(88,32){$A'$}
\put(50,1){$A$}
\put(87.5,8){$B$}
\put(87,83){$B'$}
\put(0,89){\huge $X$}
\end{overpic}
\caption{$q=1$.}
\label{fig:typeS_1}
\end{subfigure}
\begin{subfigure}{.32\linewidth}
\centering
\begin{overpic}[scale=.12, percent]
{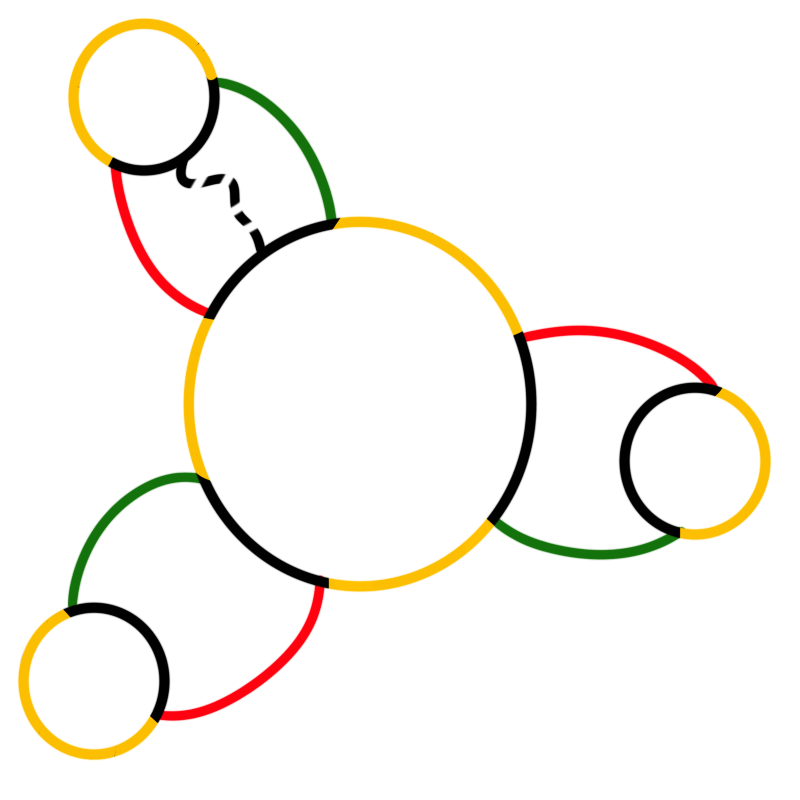}
\put(42,45){\Large $V$}
\put(82,37){\large $V'$}
\put(7,9){\large $V'$}
\put(12,83){\large $V'$}
\put(68,21){$A'$}
\put(70,60){$A$}
\put(40,80){$A'$}
\put(90,26){$B'$}
\put(55,70){$B$}
\put(80,85){\huge $X$}
\end{overpic}
\caption{$q>1$.}
\label{fig:typeS_2}
\end{subfigure}
\caption{Handlebody-knots of types $\Mo$ and $\Se$.}
\end{figure}
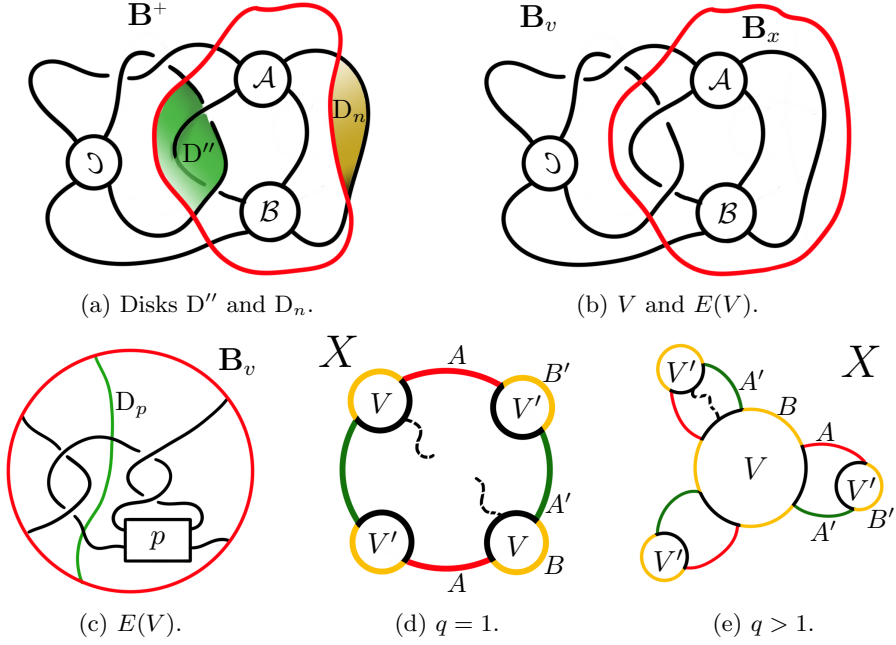

 
Suppose $\pair$ is of type $\Se$. 
Then its JSJ-graph is \ref{fig:201h} or \ref{fig:a301hi} by Lemma \ref{lm:typeKMS}. 
Let $A,A'$ be the two type $3$-$3$ annuli corresponding to the edges of the bigon.
Then by \cite[Lemma $3.7$, Corollaries $3.5$, $3.10$]{Wan:22p} they are of the same slope $(\frac{p}{q},pq)$ 
with $p\neq 0,\pm 1$, $q>0$.


\begin{theorem}\label{teo:typeS} 
Let $A_1,A_2$ be the two non-characteristic annuli in $\Compl\HK$.  
\begin{enumerate}[label=\textnormal{(\roman*)}]
\item If $q=1$, then $A_1,A_2$ either both are of type $3$-$2$i or both are of type $3$-$2$ii. 
\item If $q>1$, then one of $A_1,A_2$ is of type $3$-$2$i
and the other of type $3$-$2$ii. 
\end{enumerate}
\end{theorem}
\begin{proof}
Let $X$ be the Seifiert solid torus cut off by 
$A,A'$ from $\Compl\HK$, and 
$B,B'$ the two annuli in $\overline{\partial X-(A \cup A')}$. Since $\partial A$, $\partial A'$ 
are parallel in $\partial\HK$, and $A,A'$ 
are of type $3$-$3$, there exists an essential separating disk $D_a\subset\HK$ disjoint from $B,B'$ and separating them. 
Denote by $V,V'\subset \HK$ the two solid tori cut off by $D_a$ containing $B,B'$, respectively; it may be assumed that $\partial A \cap V$ (resp.\ $\partial A \cap V'$) has a slope of $\frac{p}{q}$ 
(resp.\ $pq$) with respect to $(\sphere,V)$ (resp.\ 
$(\sphere,V')$).

Observe now that the two non-characteristic annuli $A_1,A_2$ can be identified with 
the frontier of a regular neighborhood of $A \cup B'\cup A'$ in $X$ and $A\cup B\cup A'$, respectively.
In particular, we have $\partial A_1\subset V$ and $\partial A_2\subset V'$. Thus, by Lemma \ref{lm:cri_threetwo_ii}, 
to determine the types of $A_1,A_2$, it amounts to check whether  
$(\sphere,cV),(\sphere,cV')$ are trivial, respectively. 

In the case $q=1$, the two knots $(\sphere,cV),(\sphere, cV')$ are equivalent (see Fig.\ \ref{fig:typeS_1}), and hence the first assertion. In the case $q>1$, the union $M:=V\cup A \cup A'\cup X$ is a Seifert bundle with two exceptional fiber since $p\neq 0,\pm 1$;
hence by the classification of Seifiert structure on $\sphere$
\cite{Sei:33}, $(\sphere,cV)$ is trivial and $(\sphere,cV')$ is a $(p,q)$-torus knot (see Fig.\ \ref{fig:typeS_2}); this implies the second claim. 
\end{proof}


\subsection{Arcs in a 4-punctured sphere}\label{subsec:arcs_punctured_sphere}
We collect results on arcs in a $4$-punctured sphere needed in our investigation on type $\Kl$ handlebody-knots.   
Let $P$ be an oriented $4$-punctured sphere $P$, and $\Ce,\Co$ two components  of $\partial P$ with the induced orientation, and $e\in \Ce, o\in\Co$. In the following, we define a coordinate for each oriented arc $\gamma$ from $e$ to $o$. 

We use the convention: given an oriented surface $S$ and oriented arcs $k,l\subset S$,
and a point $b\in k\cap l$; if the orientation $[k,l]_b$ 
induced by $k,l$ at $b$ coincides with 
(resp.\ differs from) the orientation of 
$S$, then we write $[k,l]_b=+1$ (resp.\ $-1$), and say $k,l$ are positively (resp.\ negatively) 
oriented at $b$. We denote by $\algint(k,l)$ 
the algebraic intersection number $\sum\limits_{b\in k\cap l}[k,l]_b$ of $k,l$, and set $\geoint(k,l):=\vert k\cap l\vert$.

\subsubsection{Coordinate system}  
Denote by $\Ce',\Co'\subset \partial P$ the two components other than $\Ce,\Co$, and 
choose an oriented arc $d_e$ (resp.\ $d_o$) going from $\Ce'$ to $\Ce$ (resp.\ $\Co'$ to $\Co$), and an oriented arc 
$s_0$ going from $\Ce$ to $\Co$. 
The the triplet $\mathfrak{C}:=\{d_e,d_o,s_0\}$ is called
a \emph{coordinate system} for $\{P,\Ce,\Co\}$.
Note that $d_e,d_o,s_0$ determines, up to isotopy without fixing endpoints, an unique arc $s_0'$ going from $\Ce'$ to $\Co'$ and disjoint from 
$d_e\cup d_o\cup s_0$. 

\subsubsection{Slope of arcs}
Consider the integrally-punctured plane 
\[\pR:=\mathbb{R}^2-\smashoperator{\bigcup_{\bm\in\mathbb{Z}^2}} \mathring{\geodisk}_\epsilon(\bm),\] 
where $\geodisk_\epsilon(\bm)$ is the disk of radius $\epsilon\ll\frac{1}{2}$ with center at $\bm$.
Let $\bt$ be the generator of $\mathbb{Z}_2$, 
and consider the map from 
$\psi$ from $\mathbb{Z}_2$ to the homeomorphism group
$\homeo{\pR}$ given by 
\begin{align*}
\psi_i:=\psi(\bt^i):\pR &\longrightarrow \pR\\
\bv &\mapsto (-1)^i\bv,
\end{align*}
where $i=0,1$.
The map $\psi$ restricts to a homomorphism 
$\phi:\mathbb{Z}_2\rightarrow \Aut\Ztwo$ which gives us the semiproduct $\Ztwo\rtimes_\phi \mathbb{Z}_2$ and the group action of $\Ztwo\rtimes_\phi \mathbb{Z}_2$ on $\pR$  
\begin{align*}
\Ztwo\rtimes_\phi \Ztwo &\times \pR \rightarrow \pR\\
(\bm \bt^i&, \bv)\longmapsto 2\bm+\psi_i (\bv).
\end{align*}
%
%
The quotient space of $\pR$ by the group action is homeomorphic to $P$. Denote by $\geodisk_\epsilon(i,j)$ the disk with center at $\bm=(i,j)\in\Ztwo$. 
Then we can choose an orientation-preserving covering map $\pi_\circ:\pR\rightarrow P$ so that  
\begin{align*}
\tilde \Ce:=\pi_\circ^{-1}(\Ce)
=\smashoperator{\bigcup_{i,j\ \text{even}}}\partial \geodisk_\epsilon(i,j),
&\quad
\tilde C'_e:=\pi_\circ^{-1}(\Ce')
=\smashoperator{\bigcup_{\substack{
i\ \text{even},\\
j\ \text{odd}
}
}}
\partial \geodisk_\epsilon(i,j),\\
\tilde \Co:=\pi_\circ^{-1}(\Co)
=\smashoperator{\bigcup_{\substack{
i\ \text{odd},\\
j\ \text{even}
}
}}\partial \geodisk_\epsilon(i,j),
&\quad
\tilde C'_o:=\pi_\circ^{-1}(\Co')
=\smashoperator{\bigcup_{i,j\ \text{odd}}}
\partial \geodisk_\epsilon(i,j), 
\end{align*}   
\begin{equation*}
\tilde e:=\picirc^{-1}(e)
=\{(i\pm\epsilon,j)\mid i,j\ \text{even}\} \text{ and } 
\tilde o:=\picirc^{-1}(o)
=\{(i\pm\epsilon,j)\mid i\ \text{odd}, j\ \text{even}\}.
\end{equation*}
Further, it may be assumed that $\tilde d_e:=\picirc^{-1}(d_e)$ (resp.\ $\tilde d_o:= \picirc^{-1}(d_o)$) is the segment going from $(i,j\pm\epsilon)$
to $(i,j\pm (1-\epsilon))$ with $j$ odd and $i$ even (resp.\ $i$ odd) integers, and $\tilde{s}_0:=\picirc^{-1}(s_0)$ (resp.\ $\tilde{s}_0':=\picirc^{-1}(s_0')$) is the segment going from 
$(i\pm\epsilon,j)$ to $(i\pm (1-\epsilon),j)$ with $i$ even and $j$ even (resp.\ $j$ odd) integers; see Fig.\ \ref{fig:punctured_plane_fund}.  

Observe that, given an oriented arc $\alpha$ 
from $e$ to $o$,
if the lifting of $\alpha$ based at 
$(i+\epsilon,j)$ in $\pR$ goes to $(i\pm\epsilon+t_x,j+t_y)$ for some $(i,j)\in\Ztwo$, then the same holds for every $(i,j)\in\Ztwo$, and moreover, 
the lifting of $\alpha$ based at 
$(i-\epsilon,j)$ in $\pR$ goes to $(i\mp\epsilon-t_x, j-t_y)$, for every $(i,j)\in\Ztwo$. Therefore, the rational number $\frac{t_y}{t_x}$ is independent of the choice of lifting, and we call it the \emph{slope} of 
$\alpha$, with respect to $\mathfrak{C}$; see Fig.\ \ref{fig:punctured_plane_fund} for an lifting of an arc of slope $\frac{2}{3}$. 

Note that $t_y$ is always even and $t_x$ odd, and two arcs 
of the same slope differ by 
some Dehn twists along $\Ce, \Co$. Denote by $\eoQ$ 
the set of rational numbers with even numerator and odd denominator.

\begin{figure}[b]
\centering
\begin{subfigure}[t]{.5\linewidth}
\centering
\includegraphics[scale=.337]{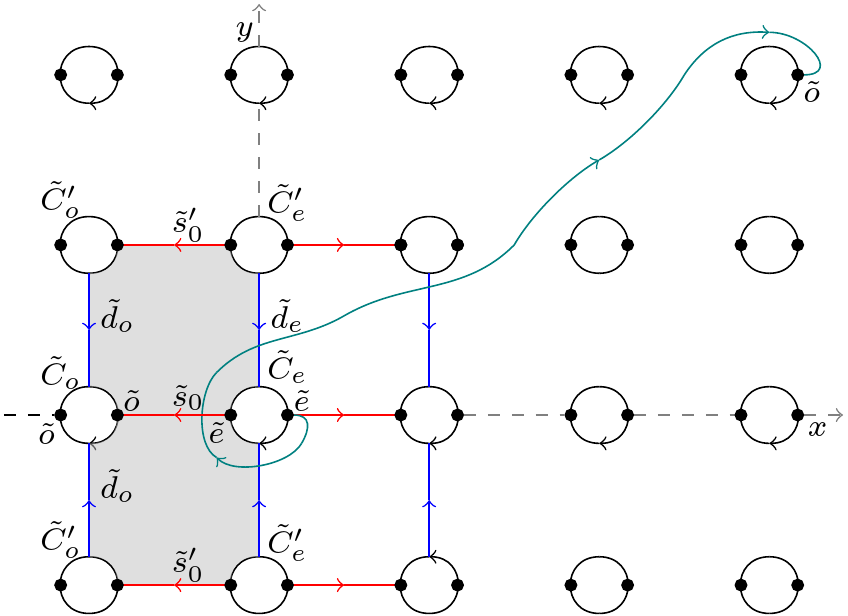}
\caption{Fundamental domain and liftings in $\pR$.}
\label{fig:punctured_plane_fund}
\end{subfigure}
\hspace*{.25cm}
\begin{subfigure}[t]{.46\linewidth}
\centering
\includegraphics[scale=.329]{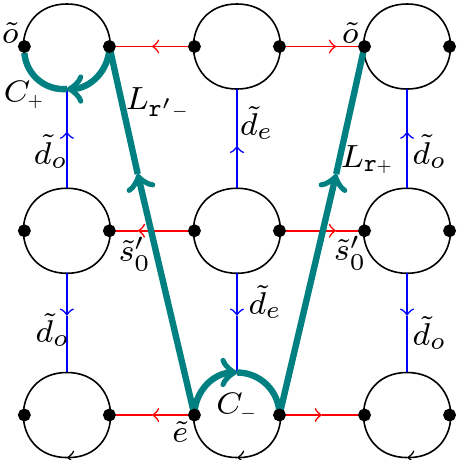}
\caption{$C_\plusminus$; $L_{\slope\plus},\slope>0$; $L_{\slope'\minus}, \slope'<0$.}
\label{fig:line_circle_model}
\end{subfigure}
\end{figure}

\begin{figure}[b]
\begin{subfigure}[t]{.32\linewidth}
\centering
\includegraphics[scale=.18]{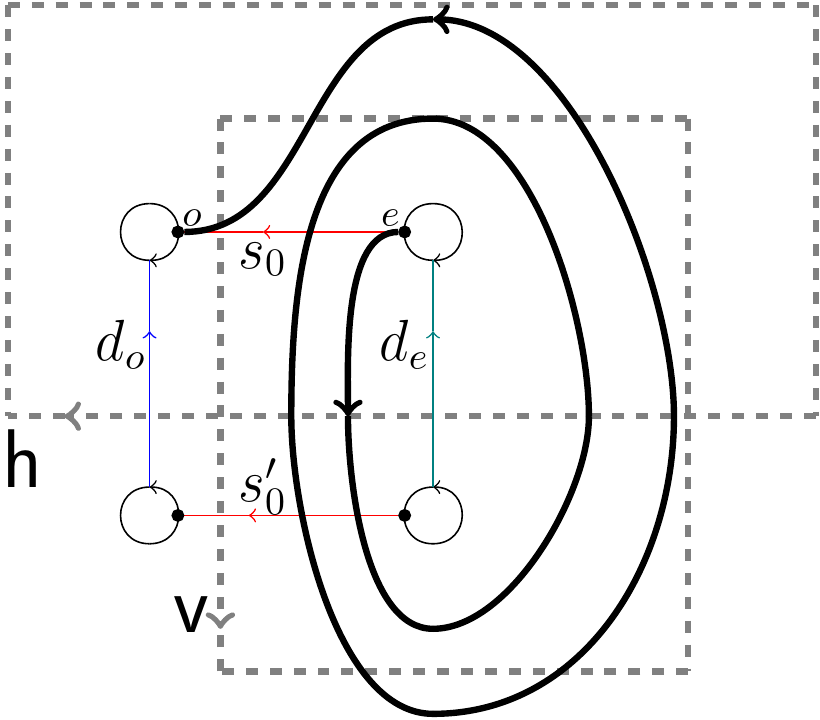}
\caption{$\rho=2,\beta=0$.}
\label{fig:null_beta}
\end{subfigure}
\begin{subfigure}[t]{.32\linewidth}
\centering
\includegraphics[scale=.18]
{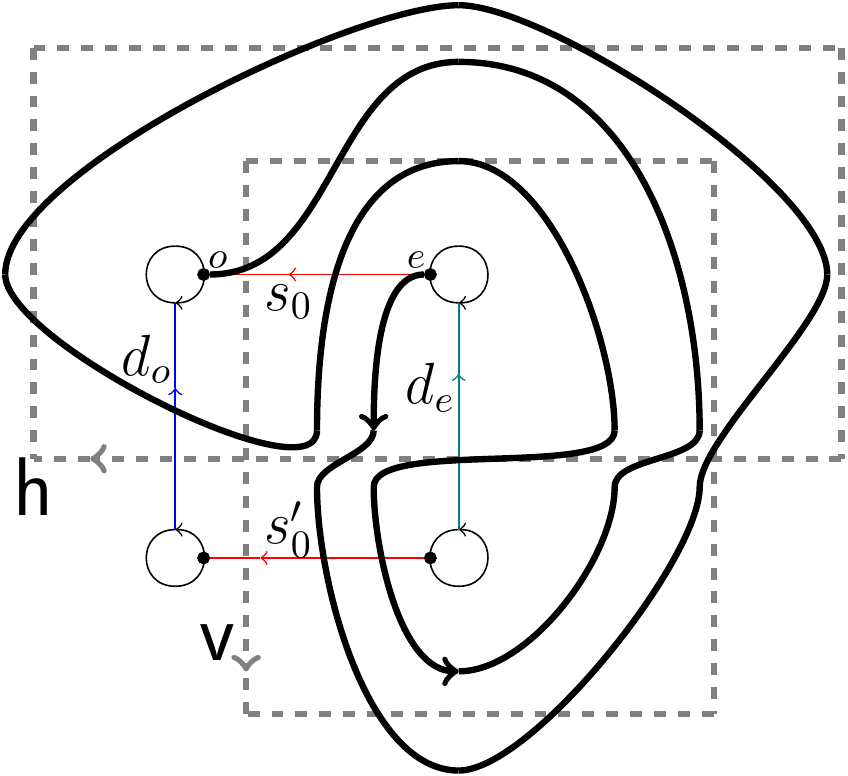}
\caption{$\rho=2,\beta=1$.}
\label{fig:positive_beta}
\end{subfigure}
\begin{subfigure}[t]{.32\linewidth}
\centering
\includegraphics[scale=.18]
{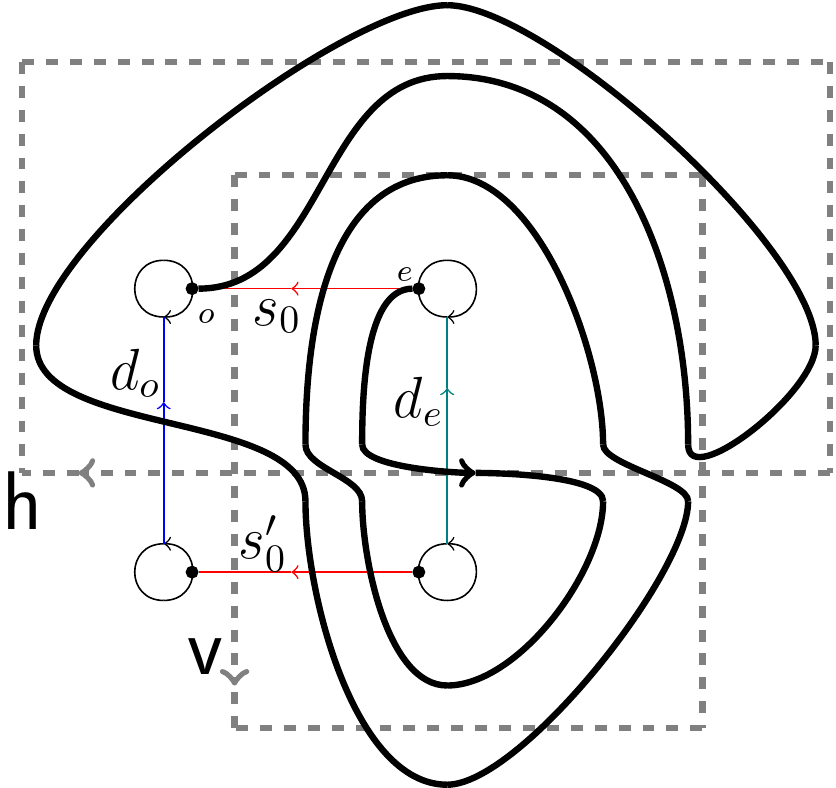}
\caption{$\rho=2,\beta=-1$.}
\label{fig:negative_beta}
\end{subfigure}
\caption{Arcs in $P$ and their lifts in $\pR$.}
\end{figure}

\subsubsection{Coordinate of arcs}\label{subsubsec:coor}  
For each $\slope\in\eoQ$, we construct 
a reference arc $s_\slope$ from $e$ to $o$
of slope $\slope$ with respect to 
$\mathfrak{C}$ as follows:
Let $\ver\subset P$ (resp.\ $\hor\subset P$) be an oriented essential circles disjoint from $d_e\cup d_o$ (resp.\ $s_0\cup s_0'$) that meets $s_0, s_0'$ (resp.\ $d_e,d_o$) each at one point with $[s,\ver]=1$, $s=s_0$ or $s_0'$ (resp.\ 
$[d,\hor]=1$, $d=d_e$ or $d_0$).

Consider first the case $\slope=2\rho\geq 0$. 
Then $s_\slope$ is given by twisting $s_0$ along $v$ $\rho$ times (see Fig.\ \ref{fig:null_beta}).
For the general case $\slope=\frac{2\rho}{2\beta+1}$ with $\rho,\beta\in \mathbb{Z}$ and $\rho\geq 0$, we construct $s_\slope$ 
as follows: Identify $P$ 
with the $2$-sphere $x^2+y^2+z^2=1$ in $\mathbb{R}^3$ with the interior of four disjoint small geodesic disks 
centered at 
$\frac{1}{\sqrt{2}}(0,\pm 1, \pm 1)$ 
removed. It may be assumed that $e,o$ 
is in the boundary of the geodesic disks 
centered at $\frac{1}{\sqrt{2}}(0,1,1), \frac{1}{\sqrt{2}}(0,-1,1)$, respectively.
Identify the loop $\hor$ with 
the equator $\{(x,y,0)\mid x^2+y^2=1\}$, and set $\bar \slope:=2\rho$. Then it may be assumed that $s_{\bar \slope}$ meets $\hor$ 
at the $2\rho$ points: $\left(\sin(\frac{\pi}{2\rho}+\frac{\pi}{\rho}i),\cos(\frac{\pi}{2\rho}+\frac{\pi}{\rho}i), 0\right)$, $i=0,\dots,2\rho-1$. 
Twisting $s_{\bar\slope}$ along $\hor$ by 
an angle of $\frac{\beta\pi}{\rho}$,
we obtain $s_\slope$; see Figs.\ \ref{fig:positive_beta}, \ref{fig:negative_beta}.

Now, given an oriented arc $\gamma$ from $e$ to $o$ with slope $\slope$, 
then, up to isotopy, 
it can be obtained by twisting 
$s_\slope$ along $C_e$ $\lambda$
times and $C_o$ $\mu$ times, for some 
$\lambda,\mu\in\mathbb{Z}$. 
The triplet $(\slope, \lambda, \mu)$ 
is then called the \emph{coordinate} of $\gamma$ with respect to $\mathfrak{C}$; 
see Fig.\ \ref{fig:punctured_plane_fund} for a lifting of an arc with coordinate $(\frac{2}{3},2,1)$.

\subsubsection{Generating loops} 
In order to decompose arcs in $P$, 
here we regard an oriented arc 
as an image of some path, and use the
same notation to denote both the arc and path. 
Let $\gamma$ be an arc from $e$ to $o$ with coordinate $(\slope,\lambda,\mu)$, where 
$\slope=\frac{2\rho}{2\beta+1}$
for some $\rho,\beta\in\mathbb{Z}$ with 
$\rho\geq 0$ and $2\rho, 2\beta+1$ relatively prime when $\slope\neq 0$. 
Consider the loop $\hat C_o:=s_0 C_o s_0^{-1}$ and the loop $\hat \ver$ 
defined as the loop in $\ver \cup s_0$ 
based at $e$. Then the homotopy classes of $C_e,\hat C_o,\hat v$ 
generate $\pi_1(P,e)$, so every loop based at $e$ is homotopic, relative to $e$, to 
a product of finitely many copies of $C_e$, $\hat C_o$ and $\hat \ver$. 

To write down the product precisely, we first note that by the definition of the coordinate, 
$\gamma\simeq C_e^\lambda s_\slope C_o^\mu\simeq  C_e^\lambda s_\slope s_0^{-1} \hat{C}_o^\mu s_0$.
Secondly, observe that $d_e,d_o,s_0'$ 
are dual to $C_e,\hat C_o,\hat \ver$.
Consider now the covering space $\pR\rightarrow P$, 
and note that $\tilde{e}$ (resp.\ $\tilde{o}$) splits the boundary component 
$\partial\geodisk(i,j)$, $i,j$ even (resp.\ $i$ odd, $j$ even) into 
the upper semicircle $C_\minus(i,j)$ and 
the lower semicircle $C_\plus(i,j)$ 
oriented so that $C_\plusminus(i,j)$ starts from $(i\pm \epsilon,j)$. 
Denote by $L_{\slope\plusminus}(i,j)$ the straight line starting from $(i\pm \epsilon,j)$ to $(i\pm (2\beta+1)\mp\epsilon,j\pm 2\rho)$ (resp.\ to
$(i\mp (2\beta+1)\mp\epsilon,j\mp 2\rho)$) if $\beta>0$ (resp.\ if $\beta<0$). 
The pair $(i,j)$ in $L_{\slope\plusminus}(i,j), C_\plusminus(i,j)$ 
is often suppressed for the sake of simplicity.

Then the lifting of 
$s_\slope$ in $\pR$ starting 
at $(i\pm\epsilon,j)$
is homotopic, with endpoints fixed, to $L_{\slope\plusminus}(i,j)$ when $\beta\geq 0$ and to 
$C_\minusplus^{-1}(i,j)L_{\slope\minusplus}(i,j) C_\plusminus(i,j)$
when $\beta<0$ (see Fig.\ \ref{fig:line_circle_model}). 
In particular, $s_\slope$ meets $d_o$ and $d_e$ each $\beta$ times and in an alternating manner, and 
points in $\gamma\cap s_0'$ 
are interpolated in between points in $\gamma\cap d_o$ and points $\gamma \cap d_e$. 
This leads us to the following.

\subsubsection{Alternating functions} 
A \emph{unit sequence} is a finite sequence with each term either $1$ or $-1$.
Let $\tau$ be a non-negative integer.
Then a \emph{paired unit sequence} $A_\tau$ is a unit sequence $\{\upsilon_i\}_{i=1}^{2\tau}$
of length $2\tau$.
Given a paired unit sequence $A_\tau$, the \emph{induced alternating function} $\mathcal{A}_\tau^G$ on an $H$-group $G$ is 
the function $\mathcal{A}_{\tau}^G:G\times G\rightarrow G$ 
given by
\[\mathcal{A}_{\tau}(x,y)= x^{\upsilon_1} y^{\upsilon_2}
\dots x^{\upsilon_{2\tau-1}}y^{\upsilon_{2\tau}},
\]
where $x_1\dots x_t:=(\cdots((x_1x_2)x_3)\cdots x_t)$ with juxtaposition denotes the multiplication in $G$.

\cout{
Given a paired unit sequence $A_\tau=\{\upsilon_i\}_{i=1}^{2\tau}$, an extension $\hat A_\tau$ of $A_\tau$ is a unit sequence 
that contains $A_\tau$ as a subsequence. 
More precisely, if $\hat A_\tau=\{\epsilon_i\}_{i=1}^\sigma$, then 
$A_\tau=\{\upsilon_i=\epsilon_{\kappa(i)}\}_{i=1}^{2\tau}$, for some strictly increasing function $\kappa:\{1,\dots,2\tau\}\rightarrow \{1,\dots,\sigma\}$;
note that $\kappa$ is part of the definition of $\hat{A}_\tau$.
}
Given a paired unit sequence $A_\tau=\{\upsilon_i\}_{i=1}^{2\tau}$, an extension $(\hat A_\tau,\kappa)$ of $A_\tau$ consists of a unit sequence
$\hat A_\tau=\{\epsilon_i\}_{i=1}^\sigma$, $\sigma\geq 2\tau$, 
and a strictly increasing function $\kappa:\{1,\dots,2\tau\}\rightarrow \{1,\dots,\sigma\}$ such that
$A_\tau=\{\upsilon_i=\epsilon_{\kappa(i)}\}_{i=1}^{2\tau}$. For the sake of simplicity, $\kappa$ is often dropped from the notation $(\hat A_\tau,\kappa)$. 
 
Set $\kappa(0)=0,\kappa(2\tau+1)=\sigma+1$. Then for any $H$-group $G$, $\hat{A}_\tau$ induces a function, called 
the induced \emph{interpolating function}, $\mathcal{\hat A}_\tau^G$: $G\times G\times G\rightarrow G$ given by 
\[\mathcal{\hat A}_{\tau}(x,y,z)= z^{\zeta_0}x^{\upsilon_1}z^{\zeta_1} y^{\upsilon_2}
\dots z^{\zeta_{2\tau-2}}x^{\upsilon_{2\tau-1}}z^{\zeta_{2\tau-1}}y^{\upsilon_{2\tau}}z^{\zeta_{2\tau}},\]
where $\zeta_i=\smashoperator{\sum\limits_{ j=\kappa(i)+1}^{\kappa(i+1)-1}}\epsilon_j$ 
if $\kappa(i+1)>\kappa(i)+1$, and $\zeta_i=0$ otherwise, $i=0,\dots,2\tau$. 
Particularly, when $z$ is the identity $1_G$, we have $\mathcal{\hat A}_\tau^G(-,-,1_G)$
is homotopic to $\mathcal{A}_\tau^G(-,-)$. We drop $G$ from $\mathcal{A}_\tau^G,\mathcal{\hat A}_\tau^G$ when it is clear from the context; main examples of $G$ here are loop spaces and fundamental groups with discrete topology.

\subsubsection{Decomposition}
Given an oriented arc $\gamma$ from $e$ to $o$ with coordinate $(\slope,\lambda,\mu)$ where $\slope=\frac{2\rho}{2\beta+1}$, $\rho,\beta\in\mathbb{Z},\rho\geq 0$, and $2\rho,2\beta+1$ are relatively prime when $\slope\neq 0$. 
\begin{lemma}\label{lm:arc_decomposition}
The oriented arc $\gamma$ induces a paired
unit sequence $A_{\vert\beta\vert}$
and an extension $\hat{A}_{\vert\beta\vert}$ of $A_{\vert\beta\vert}$
such that 
\begin{align*}
\gamma\simeq C_e^\lambda \mathcal{\hat A}_\beta(\hat C_o, C_e, \hat \ver)\hat C_o^\mu s_0& \quad\beta\geq 0\\
\gamma\simeq C_e^\lambda \mathcal{\hat A}_{-\beta}(C_e, \hat C_o, \hat \ver)\hat C_o^\mu s_0&\quad \beta<0.
\end{align*} 
In addition, when $\beta<0$, the first and last terms of $\hat A_{-\beta}$ are in $A_{-\beta}$, and they are $-1,1$, respectively. Furthermore, 
$\mathcal{\hat A}_0(\hat C_0, C_e,\hat \ver)=\hat \ver^\rho$ and 
$\mathcal{\hat A}_{-1}(C_e,\hat C_0,\hat \ver)=C_e^{-1}\hat \ver^{-\rho} \hat C_o$. 
\end{lemma}  
Recall first $\gamma\simeq \Ce^\lambda s_\slope \Co^\mu$
and the lifting $\tilde s_{\slope\plus}$ 
of $s_\slope$ 
starting at $(i+\epsilon,j)$ is homotopic, with endpoints fixed, to 
$L_{\slope\plus}$ when $\beta\geq 0$ and to $C_\minus^{-1}
L_{\slope\minus}
C_\plus$
when $\beta<0$; therefore $s_\slope$ meets $d_e\cup d_o$ at $2\vert\beta\vert$ points: $x_1,\dots, x_{2\vert\beta\vert}$, and meets $d_e\cup d_o\cup s_0'$ 
at $\sigma:=2\vert \beta\vert+\rho$ points: $y_1,\dots, y_{\sigma}$. 
It may be assumed that they are ordered consecutively along $s_\slope$. In particular, 
there is an order-preserving injective function 
$\kappa:\{1,\dots,2\vert\beta\vert\}\rightarrow \{1,\dots, \sigma\}$ such that $x_i=y_{\kappa(i)}$, $i=1,\dots,2\vert\beta\vert$.

Now observe that if $x_i\in d_o$, then $x_{i+1}\in d_e$; 
similarly, if $x_i\in d_e$, then $x_{i+1}\in d_o$, for every $i<\sigma$; furthermore, $x_1\in d_o$ (resp.\ $x_1\in d_e$) if $\beta>0$ (resp.\ if $\beta<0$). 
In the case $\beta>0$ (resp.\ $\beta<0$), we define $\upsilon_i$ to be $[d_o,s_\slope ]_{x_i}$ 
when $i$ is odd (resp.\ even), and $[d_e,s_\slope]_{x_i}$ 
when $i$ is even (resp.\ odd). Thus $\gamma$, which determines $s_\slope$,
determines a paired unit sequence $A_{\vert\beta\vert}:=\{\upsilon_1,\dots,\upsilon_{2\vert\beta\vert}\}$. 
Similarly, for each $y_i\in u$, $u=d_e, d_o$ or $s_0'$, setting $\epsilon_i:=[u,s_\slope]_{y_i}$, we obtain an extension $\hat A_{\vert\beta\vert}:=\{\epsilon_1,\dots,\epsilon_\sigma\}$ of 
$A_{\vert\beta\vert}$. 
Since $d_e,d_o,s_0'$ are dual to 
$\Ce,\hat\Co, \hat \ver$, the loop 
$s_\slope s_0^{-1}$ is homotopic to  
\begin{align*}
\mathcal{\hat A}_\beta(\hat \Co, \Ce, \hat \ver)& \text{ if } \beta\geq 0\\
\mathcal{\hat A}_{\minus\beta}(\Ce, \hat \Co, \hat \ver) & \text{ if } \beta<0.
\end{align*}
The first assertion then follows from   
\[\gamma\simeq \Ce^\lambda s_\slope \Co^\mu\simeq  \Ce^\lambda s_\slope s_0^{-1}\hat C_o^\mu s_0.\]  
The second assertion can be seen from the fact that $\tilde s_{\slope\plus}
\simeq C_\minus^{-1}L_ {\slope\minus}  C_\plus$ 
when $\beta<0$, and $C_\minus^{-1}$ (resp. $C_\plus$) meets $\tilde d_e$ (resp.\ $\tilde d_o$)
negatively (resp.\ positively). 
For the last claim, we note 
when $\beta=0$ (resp.\ $\beta=-1$), $L_{\slope\plus}$ (resp.\ $L_{\slope\minus}$) does not meet $d_e\cup d_o$ but
meets $s_0'$ positively (negatively) $\rho$ times.
\begin{definition}
We call $A_{\vert\beta\vert},\hat A_{\vert\beta\vert}$ given in the proof of Lemma \ref{lm:arc_decomposition}, the induced paired unit sequence and extension by $\gamma$.
\end{definition}


\subsection{Type $\Kl$}
Let $\pair$ be of type $\Kl$. 
Then, as shown in Section \ref{subsec:typeK}, 
$\Compl\HK$ admits a unique non-separating annulus $A$. It is non-characteristic and its type has been determined.

\begin{lemma}{\cite[Lemma $3.14$]{Wan:22p}}\label{lm:typeK_non_separating}
The non-separating annulus $A\subset\Compl\HK$ is of type $3$-$3$ with the slope pair $(\frac{p}{q},pq)$, $q>0$, $p\neq 0,\pm 1$.
\end{lemma} 
Given a separating annulus $A\subset\Compl\HK$, then, since components of $\partial A$ are parallel in $\partial\HK$, the annulus $A$ 
determines an element $\xA$, up to conjugacy and inverse, in $\pi_1(\HK)$, and we have the following algebraic condition for $A$ to be of type $4$-$1$.  
\begin{lemma}\label{lm:cri_fourone}
If there exists no generating pair $\{x,y\}$ 
of $\pi_1(\HK)$ such that $x^n$ is conjugate to $\xA$, for some
$n>0$, then $A$ is of type $4$-$1$. 
\end{lemma}
\begin{proof}
Suppose otherwise, and $A$ is of type $3$-$2$.
Then by the definition, there exists 
a non-separating disk $\Disk A\subset\HK$ disjoint 
from $\partial A$. Let $V$ be the solid torus $\HK-\openrnbhd{\Disk A}$, and 
$D$ be a meridian disk of $V$ disjoint from 
$\rnbhd{\Disk A}$. Then $\partial A$ 
meets $D$ minimally $2n$ times,
for some $n>0$.
Let $x,y$ be the generating pair 
of $\pi_1(\HK)$ given by simple loops dual to $D, \Disk A$, respectively.
Then $x^n$ is conjugate to $\xA$.
\end{proof}
To avail oneself of Lemma \ref{lm:cri_fourone}, we use a criterion derived from \cite[Lemma $1.1$]{ChoKod:15}. 
Let $\mathbb{Z}\ast\mathbb{Z}$ be the free group of rank $2$. 
An element $x\in\mathbb{Z}\ast\mathbb{Z}$ 
is said to be \emph{primitive} 
if there exists $y\in \mathbb{Z}\ast\mathbb{Z}$ such that $\{x,y\}$ 
is a generating pair.

\begin{lemma}\label{lm:cri_not_power_pimitive}
Let $\{x,y\}$ be a generating pair of $\mathbb{Z}\ast\mathbb{Z}$. Consider the word $w=a^{\epsilon_1}b^{\eta_1}\dots a^{\epsilon_n}b^{\eta_n}$ where $n\geq 1$ and $\epsilon_i,\eta_i$ nonzero, for every $i$.
If $\epsilon_1\neq \epsilon_i$, $\eta_n\neq \eta_j$, for some $i\neq 1$, $j\neq n$, or both $\vert \epsilon_1\vert,\vert \eta_n\vert$ 
are greater than $1$,
then $w$ is not a power of some primitive element in $\mathbb{Z}\ast\mathbb{Z}$.
\end{lemma}
\begin{proof}
By \cite[Lemma $1.1$]{ChoKod:15},
it suffices to show that $w$ satisfies
one of the following: $w$ contains 
both $x,x^{-1}$, or both $y,y^{-1}$, 
or both $x^{\pm 2},y^{\pm 2}$. 
The condition $\vert\epsilon_1\vert,\vert\eta_n\vert$ is greater than $1$ 
clearly satisfies the third case above, 
so we only need to consider 
the condition $\epsilon_1\neq \epsilon_i,\eta_j\neq \eta_n$. 
Suppose $w$ does not have both $x,x^{-1}$ and both $y,y^{-1}$.
Then $\epsilon_1\epsilon_i\geq 2$ and $\eta_j\eta_n\geq 2$, and thus one of $x^{\epsilon_1},x^{\epsilon_i}$ 
contains $x^2$ or $x^{-2}$
and one of $y^{\eta_j},y^{\eta_n}$ 
contains $y^2$ or $y^{-2}$. 
\end{proof} 
 
\begin{figure}[b]
\begin{subfigure}[t]{.5\linewidth}
\centering
\begin{overpic}[scale=.13,percent]
{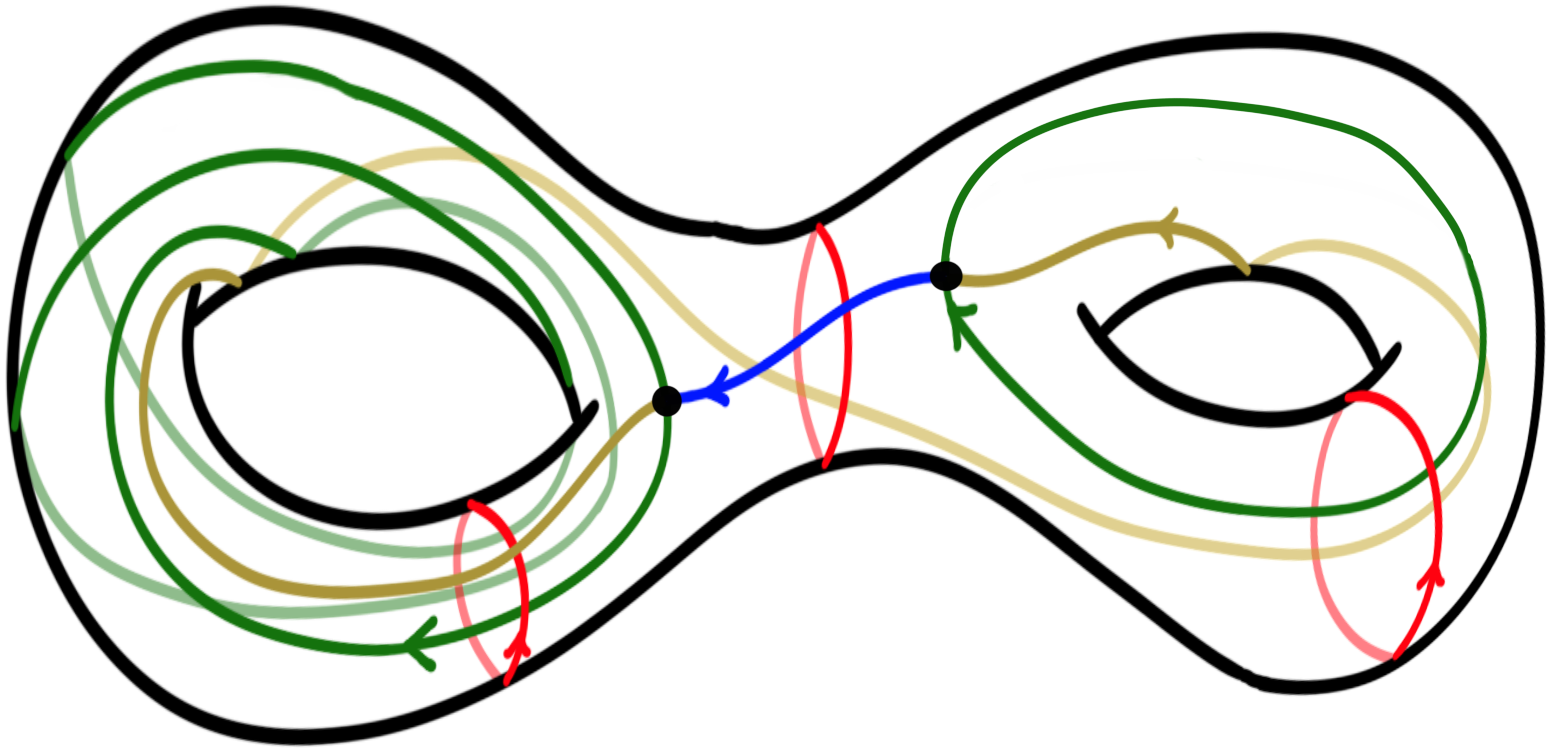}
\put(91.5,4){$d_2$}
\put(33.6,2){$d_1$}
\put(52,13){$\disk A$}
\put(91,38){$l_2$}
\put(18,2.5){$l_1$}
\put(62,32){$b_2$}
\put(43.5,18){$b_1$}
\put(55.8,25){$k_\plus$}
\put(78,34){$k_\minus$}
\end{overpic}
\caption{Arcs on $\partial\HK$.}
\label{fig:curves_hk}
\end{subfigure}
\hspace*{.02cm}
\begin{subfigure}[t]{.48\linewidth}
\centering
\begin{overpic}[scale=.15,percent]{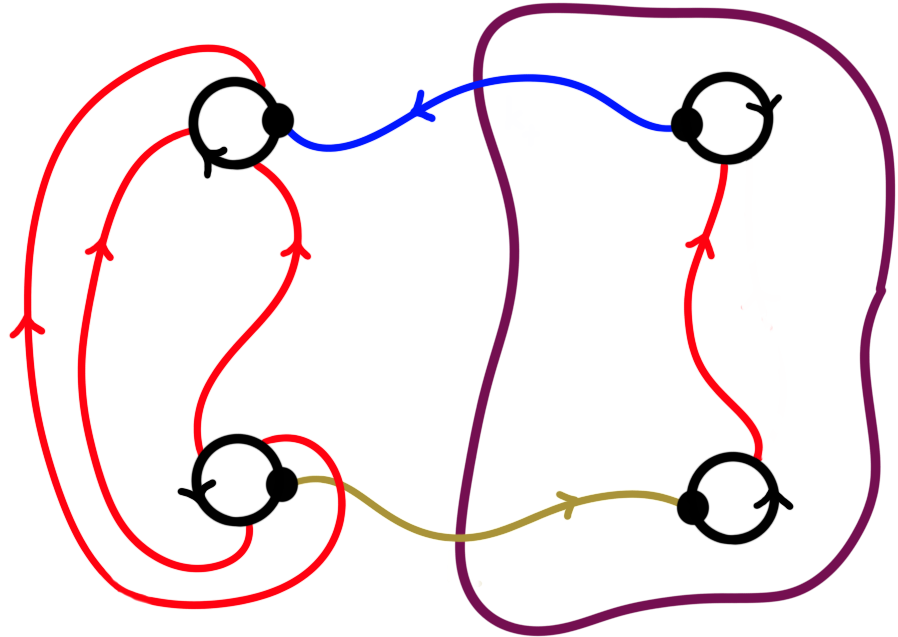}
\put(95,65){$\disk A$}
\put(3,4){$d_{11}$}
\put(28,30){$d_{12}$}
\put(11,33){$d_{13}$}
\put(80,40){$d_2$}
\put(71.7,52){\footnotesize $2_\plus$}
\put(32,58){\footnotesize$1_\plus$}
\put(73.5,9){\footnotesize $2_\minus$}
\put(31,11.3){\tiny $1_\minus$}
\put(45,51){$k_\plus$}
\put(65,17.5){$k_\minus$}
\end{overpic}
\caption{Liftings.}
\label{fig:lifting}
\end{subfigure}
\caption{Arcs on $\partial\HK$ and their lifts in $P$.}
\end{figure}

In the following, as in Section \ref{subsec:arcs_punctured_sphere}, we regard an oriented arc as the image of some path, and use the same notation for both the arc and path.

\subsubsection{Arcs and disks induced by non-characteristic annuli}\label{subsubsec:curves_nonchar}
Hereinafter, $A$ denotes the type $3$-$3$ annulus   
in $\Compl\HK$, and $\Disk A\subset\HK$ denotes an essential, 
separating disk disjoint from $\partial A$. Let $V_1,V_2$ be the solid tori cut off by
$\Disk A$ from
$\HK$; 
by Lemma \ref{lm:typeK_non_separating}, it may be assumed that components   
$l_1:=A\cap V_1$ and $l_2:=A\cap V_2$ 
of $\partial A$ have slopes of $\frac{p}{q}$ and $pq$ with respect to $(\sphere, V_1), (\sphere, V_2)$, respectively, with 
$p\neq 0,\pm 1$ and $q>0$.
Let $D_i\subset \HK$ be meridian disks of $V_i$ disjoint from $\Disk A$, $i=1,2$, and denote by $\disk A$ (resp.\ $d_1, d_2$) the boundary of $\Disk A$ (resp.\ $D_1, D_2$); see Fig.\ \ref{fig:curves_hk}. 
 
Recall from Section \ref{subsec:typeK}, 
non-characteristic separating annuli 
are classified by the frontier $A_n$ 
of the M\"obius band $M_n:=\pi^{-1}(\gamma_n)$, $n\in \mathbb{Z}$,
where $\pi:X\rightarrow \pklein$ 
is the bundle projection from 
the I-fibered component $X\subset\Compl\HK$ to a once-punctured Klein bottle $\pklein$, and $\gamma_n$ is  a simple loop, homotopic to $\alpha\beta^n$ with $\alpha,\beta$ the oriented loops in Fig.\ \ref{fig:pklein}. 
Particularly, $\partial A_n$ 
and $\partial M_n$ are parallel in
$\partial\HK$.

\subsubsection{Orientation}\label{subsubsec:orientation}
We orient $l_1,l_2$ so that they represent the same homology class in $H_1(A)$, and orient $d_1,d_2$ so that $[d_i,l_i]_b=+1$ in $\partial\HK$, for every $b\in l_i\cap d_i$, 
$i=1,2$ (Fig.\ \ref{fig:curves_hk}).
Observe that, for each $n$, $\partial M_n$ meets $\partial A$ at two points, one in $l_1$ and the other in $l_2$, and it may be assumed that $M_n$, $n\in\mathbb{Z}$, all meet $\partial A$ at the same points $b_1\in l_1,b_2\in l_2$, and $\partial M_n$ is oriented so 
that $[\partial M_n, l_2]_{b_2}=-1$ and $[\partial M_n, l_1]_{b_1}=+1$. 
Set $b_2$ to be the base point of $\partial \HK$ and $\HK$.

Now $b_1\cup b_2$ cuts $\partial M_0$ into two oriented arcs;
denote by $k_\plus$ the arc going from $b_2$ to $b_1$ 
and by $k_\minus$ the other arc (Fig.\ \ref{fig:curves_hk}). Then $\partial M_n$ 
is homotopic to $k_\plus l_1^n k_\minus l_2^n$, relative to $b_2$, where $l_i$ is regarded as an oriented loop
based at $b_i$, $i=1,2$. 
The goal is to identify the homotopy class of the loop 
$k_\plus l_1^nk_\minus l_2^n$, 
in terms of a suitable generating pair 
of $\pi_1(\HK,b_2)$, abbreviated to 
$\pi_1(\HK)$ henceforth.
\subsubsection{Liftings in a four-times punctured sphere}
Cutting $\partial\HK$ along $l_1,l_2$, 
we obtain a $4$-punctured sphere and  
a quotient map $\pi:P\rightarrow \partial\HK$. Orient $P$ so that $\pi$
is orientation-preserving. By convention, a lifting of an oriented arc from $\partial\HK$ to $P$ always carries the induced orientation.

The lifting of $d_2\subset \partial\HK$ to $P$ is an oriented arc, denoted by the same letter $d_2$, whereas the lifting of $d_1\subset \partial\HK$ to $P$ consists of $q$ oriented arcs, denoted by $d_{1i}$, $i=1,\dots, q$; set $d_{1\ast}:=d_{11}\cup\dots\cup d_{1q}$. 
The lifting of $\disk A\subset\partial\HK$ to $P$ is also denoted by the same letter $\disk A$; as yet no orientation is assigned to $\disk  A\subset\partial\HK$ and hence to its lifting in $P$. 
Likewise, the same notation $k_{\plusminus}$ is used for 
the lifting of $k_{\plusminus}\subset\partial\HK$ to $P$. 
Denote by $2_\plus$ (resp.\ $1_\minus$) the starting point of $k_\plus$ (resp.\ $k_\minus$) and its endpoint by $1_\plus$ (resp.\ $2_\minus$).
Also, let $l_{i\plusminus}\subset \partial P$ be the component  
containing $i_\plusminus$, $i=1,2$. We then have $l_{i\plusminus}=\pi^{-1}(l_i)$, $i=1,2$, and $d_2$ (resp.\ $d_{1\ast}$) in $P$ goes from $l_{2\minus}$ to $l_{2\plus}$ (resp.\ from $l_{1\minus}$ to $l_{1\plus}$);
see Fig.\ \ref{fig:lifting} for an example with $q=3$.

\subsubsection{Coordinate system}
By convention, we order 
$d_{11},\dots, d_{1q}$ so that $l_{1\minus}$, 
as an oriented loop based at $1_\minus$, 
meets $d_{1*}$ in consecutive order; 
particularly, $d_{1\ast}$ cuts $P$ into $q$ components with $1_\minus$ in the one bounded by $d_{1q}$ and $d_{11}$. 
Fix an arc $s_\plus$ (resp.\ $s_\minus$) going from $2_\plus$ to $1_\plus$ (resp.\ from $1_\minus$ to $2_\minus$) disjoint from $d_{1q}, d_2$ that meets $\disk A$ at one point. The triplet $\{d_2,d_{1q},s_\plus\}$ gives a 
coordinate system $\mathfrak{C}$ 
for $(P,l_{2\plus},l_{1\plus})$. 
Note that $s_\plusminus$ are not unique, and every two choices of $s_\plus$ (resp.\ $s_\minus$) differ by some Dehn twists along $\disk A$. Set $s:=s_\plus\cup s_\minus$, and 
denote by the same letters the images of $s_\plus,s_\minus,s$ in $\partial\HK$ under $\pi$. 
Suppose the component containing $1_\plus $ is bounded by $d_{1\delta}$ and $d_{1\delta+1}$. 
Then the construction of $s_\plusminus$ implies the following. 
\begin{lemma}\label{lm:basic_loop}\hfill
\begin{enumerate}
\item $s\cap l_i=b_i$, $i=1,2$, with $[s,l_1]_{b_1}=1$ and $[s,l_2]_{b_2}=-1$.
\item $s\cap d_2=\emptyset=s\cap d_{1q}$
\item $\algint(d_1,s)=\delta$, and $p\delta\equiv -1$ (mod $q$). 
\item $0\leq \delta <q$, and
$q>0$ if and only if $\delta\neq 0$.
\item If $q=2$, $\delta=1$. 
\end{enumerate}
\end{lemma}

\subsubsection{Homotopy classes of $k_\pm$}
Let $(\slope,\lambda_\plus, \mu_\plus)$ be the coordinate of $k_\plus$ with respect to $\mathfrak{C}$, where $\slope=\frac{2\rho}{2\beta+1}$, $\rho, \beta\in\mathbb{Z}$,
$\rho\geq 0$, and $2\rho,2\beta+1$ relatively prime when $\slope\neq 0$. Denote by $A_{\vert \beta\vert},\hat{A}_{\vert\beta\vert}$ the paired unit sequence and its extension induced by $k_\plus$, respectively, and set 
$\hat l_{1\plus}:=s_\plus l_{1\plus}s_\plus^{-1}$.
Orient $d_A$ so that $s_\plus,d_A$ are positively oriented at $s_\plus\cap d_A$, and consider the oriented loop  
$\hat a_{2\plus}\subset \disk A\cup s_\plus$ based at 
$2_\plus$. 

\begin{lemma}\label{lm:kplus}
The path $k_\plus$ 
is homotopic, relative to $2_\plus\cup 1_\plus$, to
\begin{align*}
l_{2_\plus}^{\lambda_\plus}\mathcal{\hat A}_\beta(\hat l_{1\plus}, l_{2\plus}, \hat a_{2\plus})\hat l_{1\plus}^{\mu_\plus}s_\plus & \text{ if } \beta\geq 0, \text{ and to }\\
l_{2_\plus}^{\lambda_\plus}\mathcal{\hat A}_{\minus\beta}(l_{2\plus}, \hat l_{1\plus}, \hat a_{2\plus})\hat l_{1\plus}^{\mu_\plus}s_\plus & \text{ if } \beta<0.
\end{align*} 
\end{lemma}
\begin{proof}
Replace 
$\gamma, e,o,C_e,C_o,s_0, \ver$ in Lemma \ref{lm:arc_decomposition} with $k_\plus, 2_\plus,1_\plus, l_{2\plus}, l_{1\plus}, s_\plus, \disk A$.   
\end{proof}
Similarly, set $\hat l_{2\minus}:=s_\minus l_{2\minus}s_\minus^{-1}$, and 
let $\hat a_{1\minus}$ be the oriented loop in $\disk A\cup s_\minus$ based at $1_\minus$.   

\begin{lemma}\label{lm:kminus}
The path $k_\minus$ is homotopic, relative to $1_\minus\cup 2_\minus$, to 
\begin{align*}
l_{1_\minus}^{\mu_\minus}\mathcal{\hat A}_\beta(\hat l_{2\minus}^{-1}, l_{1\minus}^{-1}, \hat a_{1\minus}^{-1}) \hat l_{2\minus}^{\lambda_\minus}s_\minus & \text{ if } \beta\geq 0, \text{ and to }\\
l_{1_\minus}^{\mu_\minus}\mathcal{\hat A}_{-\beta}(l_{1\minus}^{-1}, \hat l_{2\minus}^{-1}, \hat a_{1\minus}^{-1}) \hat l_{2\minus}^{\lambda_\minus}s_\minus& \text{ if } \beta<0,
\end{align*}
for some $\lambda_\minus,\mu_\minus\in\mathbb{Z}$.
\end{lemma}
\begin{proof}
Consider the involution $T:P\rightarrow P$ that sends 
$l_{2\plusminus}$ to $l_{1\minusplus}$, and $d_2$ to $d_{1q}$, and $s_\plus$ to $s_\minus$. 
Then $T(k_\minus)$ and 
$k_\plus$ have the same slope, and hence, up to homotopy relative to $1_\minus\cup 2_\minus$, $k_\minus$ can be obtained from $T(k_\plus)$ by some Dehn twists along $l_{1\minus},l_{2\minus}$.  In other words, $k_\minus$ is homotopic, relative to $1_\minus\cup 2_\minus$, to 
$l_{1\minus}^{\mu_d}T(k_\plus) l_{2\minus}^{\lambda_d}$, 
for some $\mu_d,\lambda_d\in\mathbb{Z}$.
Set $\mu_\minus:=\mu_d-\lambda_\plus$ and $\lambda_\minus:=\lambda_d-\mu_\plus$. Then the assertion follows from Lemma \ref{lm:kplus} and the fact that $T(l_{2\plus})=l_{1\minus}^{-1}$, 
$T(\hat l_{1\plus})=\hat l_{2\minus}^{-1}$, $T(\hat a_{2\plus})=\hat a_{1\minus}^{-1}$, and $T(s_\plus)=s_\minus$.
\end{proof}

Set $\hat l_{1\minus}:=s_\minus^{-1} l_1 s_\minus$, and $\hat a_{2\minus}:=s_\minus^{-1} \hat a_{1\minus} s_\minus$. Then Lemma \ref{lm:kminus} implies a decomposition of $k_\minus$ in terms of loops based at $2_\minus$. 

\begin{corollary}\label{cor:kminus}
The path $k_\minus$ is homotopic, relative to $1_\minus\cup 2_\minus$, to
\begin{align*}
s_\minus \hat l_{1\minus}^{\mu_\minus}\mathcal{\hat A}_\beta(l_{2\minus}^{-1}, \hat l_{1\minus}^{-1}, \hat a_{2\minus}^{-1})l_{2\minus}^{\lambda_\minus} 
&\text{ if } \beta\geq 0, \text{ and to }\\
s_\minus \hat l_{1\minus}^{\mu_\minus}\mathcal{\hat A}_{-\beta}(\hat l_{1\minus}^{-1}, l_{2\minus}^{-1},\hat a_{2\minus}^{-1})l_{2\minus}^{\lambda_\minus}
&\text{ if } \beta< 0.
\end{align*} 
\end{corollary}
\begin{proof}
By the definition of $\hat l_{1\minus}, \hat a_{2\minus}$, and $\hat l_2$.
\end{proof}

\subsubsection{Conjugate class of $\partial M_n$}
Set $\mu:=\mu_\plus+\mu_\minus, \lambda:=\lambda_\plus+\lambda_\minus$, 
and $\hat l_1:=s_\plus l_1 s_\plus^{-1}$
Also, choose a loop $l_0\subset \partial\HK$ based at $b_2$
and disjoint from $d_2$ so that $l_0$ 
is the preferred longitude 
of the solid torus $V:=\HK-\openrnbhd{D_2}\subset \sphere$ with $\geoint(d_1,l_0)=\algint(d_1,l_0)=1$.
Recall that $\delta:=\algint(d_1,s)$ and 
$\partial M_n$ is homotopic, 
relative to $b_2$, to $k_\plus l_1^n k_\minus l_2^n$.

Denote by $u,v$ the homotopy classes in $\pi_1(\HK)$
represented by $l_2,l_0$, respectively.
Then $\{u,v\}$ is a generating pair of $\pi_1(\HK)$ with $[\hat l_1]=v^q$ and $[s]=v^\delta$ by Lemma \ref{lm:basic_loop}. In what follows, 
by $l \simeq l'$ (resp.\ $l \simeq_\partial l'$), we understand 
two loops $l,l'$ based at $b_2$ are homotopic, relative to $b_2$, in $\HK$ (resp.\ in $\partial\HK$).   
\begin{lemma}\label{lm:homotopy_class_Mn}
If $\beta\geq 0$ (resp.\ $\beta<0$), 
then the homotopy class represented by $\partial M_n$ is conjugate to 
\begin{multline*}
\mathcal{A}_\beta(v^q,u)v^{q(n+\mu)+\delta}\mathcal{A}_\beta(u^{-1},v^{-q})u^{\lambda+n}\\
\quad \left(\text{resp. }  
\mathcal{A}_{-\beta}(u,v^q)v^{q(n+\mu)+\delta}\mathcal{A}_{-\beta}(v^{-q},u^{-1})u^{\lambda+n}
\right). 
\end{multline*}  
\end{lemma}
\begin{proof}
Let $\Phi$ be the composition $(P,2_\plusminus)\xrightarrow{\pi} (\partial\HK,b_2)\hookrightarrow (\HK,b_2)$. Then $\Phi(l_{2\plusminus})=l_2$, $\Phi(\hat l_{1\plus})=\hat l_1$, and 
$\Phi(\hat a_{2\plusminus})\simeq \mathbf{1}_{b_2}$, where 
$\mathbf{1}_{b_2}$ is the constant loop at $b_2$. 
Also, from 
\[
\Phi(\hat l_{1\minus})=s_\minus^{-1}l_1 s_\minus^{-1}
\bsimeq s^{-1}s_\plus l_1 s_\plus^{-1} s 
\bsimeq s^{-1}\hat l_1 s 
\] 
and 
$s \simeq l_0^\delta$,   
it follows that $\Phi(\hat l_{1\minus})\simeq \hat l_1$.   
Therefore, when $\beta\geq 0$ (resp.\ $\beta<0$),
\begin{multline*}
k_\plus\simeq l_2^{\lambda_\plus}\mathcal{A}_\beta(\hat l_1, l_2)\hat l_1^{\mu_\plus} s_\plus\text{  and  }
k_\minus\simeq s_\minus \hat l_1^{\mu_\minus}\mathcal{A}_\beta(l_2^{-1}, \hat l_1^{-1})l_2^{\lambda_\minus}\\
\left(\text{resp. }  
k_\plus\simeq l_2^{\lambda_\plus}\mathcal{A}_{\minus\beta}( l_2,\hat l_1)\hat l_1^{\mu_\plus} s_\plus\text{  and  }
k_\minus\simeq s_\minus \hat l_1^{\mu_\minus}\mathcal{A}_{-\beta}(\hat l_1^{-1},l_2^{-1}) l_2^{\lambda_\minus}\right).  
\end{multline*}  
Since $\partial M_n\simeq  k_\plus l_1^n k_\minus l_2^n$, we have the following:   
\begin{multline*}
\partial M_n\simeq l_2^{\lambda_\plus}\mathcal{A}_\beta(\hat l_1, l_2)\hat l_1^{\mu_\plus} s_\plus l_1^{n}
s_\minus \hat l_1^{\mu_\minus}\mathcal{A}_\beta(l_2^{-1}, \hat l_1^{-1})l_2^{\lambda_\minus}l_2^{n} 
\\
\left(\text{resp.\ }
\partial M_n\simeq l_2^{\lambda_\plus}\mathcal{A}_{-\beta}(l_2, \hat l_1)\hat l_1^{\mu_\plus} s_\plus l_1^{n}
s_\minus \hat l_1^{\mu_\minus}\mathcal{A}_{-\beta}(\hat l_1^{-1},l_2^{-1})l_2^{\lambda_\minus}l_2^{n}
\right).
\end{multline*} 
Further, from $\hat l_1=s_\plus l_1 s_\plus^{-1}$ and $s\hat l_1\simeq \hat l_1s$, we deduce that
\begin{multline*}
\partial M_n\simeq l_2^{\lambda_\plus}\mathcal{A}_\beta(\hat l_1, l_2)\hat l_1^{\mu+n}s\mathcal{A}_\beta(l_2^{-1}, \hat l_1^{-1})l_2^{\lambda_\minus+n}  
\\
\left(\text{resp.\ }
\partial M_n\simeq l_2^{\lambda_\plus}\mathcal{A}_{-\beta}(l_2, \hat l_1)\hat l_1^{\mu +n}s\mathcal{A}_{-\beta}(\hat l_1^{-1},l_2^{-1})l_2^{\lambda_\minus+n}
\right).
\end{multline*} 
The assertion then follows from the fact that 
$[l_2]=u, [s]=v^\delta$ and $[\hat l_1]=v^q$.  
\end{proof}
 
\subsubsection{On Theorem \ref{intro:teo:nonchar}\ref{intro:itm:infinite}}
Since a handlebody-knot exterior 
admits infinitely many essential annuli 
if and only if the handlebody-knot is of type $\Kl$, the following implies
Theorem \ref{intro:teo:nonchar}\ref{intro:itm:infinite}.

\begin{theorem}\label{teo:typeK}
If $\Compl\HK$ contains infinitely many separating, non-characteristic annuli, then all but at most four of them are of type $4$-$1$. 
\end{theorem}
\begin{proof}
In view of Lemmas \ref{lm:cri_fourone},
it suffices to show that, for at most four $n$'s, $[\partial M_n]$ is conjugate to a power of some primitive element in 
$\pi_1(\HK)$. 
\subsubsection*{Case $1$: $\beta\geq 0$}
By Lemma \ref{lm:homotopy_class_Mn}, $[\partial M_n]$ is conjugate to 
\[
\mathcal{A}_\beta(v^q,u)v^{q(n+\mu)+\delta}\mathcal{A}_\beta(u^{-1},v^{-q})u^{\lambda+n}.
\]  
\subsubsection*{Case $1.1$: $\beta>0$.} 
Suppose $n$ satisfies $q(n+\mu)+\delta\neq q,0$ and $\lambda+n\neq 0,1$. Then by Lemma \ref{lm:cri_not_power_pimitive},
$[\partial M_n]$ is not conjugate to a power of some primitive element, and 
hence $A_n$ is of type $4$-$1$ by Lemma \ref{lm:cri_fourone}. 
Note that if $q>1$, then $0<\delta<q$
by Lemma \ref{lm:basic_loop}, so 
$q(n+\mu)+\delta$ can never be $q$ or $0$; 
in this case, all but at most
two separating, non-characteristic 
annuli in $\Compl\HK$ is of type $4$-$1$.

If $\beta=0$, then 
$[\partial M_n]$ 
is conjugate to 
$v^{q(\mu+n)+\delta}u^{\lambda+n}$,
and we divide the case into two situations.
\subsubsection*{Case $1.2$: $\beta=0$, $q>2$.}
This implies 
$q(\mu+n)+\delta\neq 0$ and
there exists at most one 
$n$, denoted by $n_v$ if existing, 
such that $q(\mu+n)+\delta=\pm 1$.
Therefore, for any $n$ that satisfies $n\neq n_v$ and 
$n+\lambda\neq 0,\pm 1$, $[\partial M_n]$ is not conjugate to
a power of some primitive element, and 
hence $A_n$ is of type $4$-$1$.

Before proceeding with the case $q\leq 2$, we 
first observe some implication of the condition $\beta=0$.
Denote by the same letters the images of 
$\hat l_{1\plusminus}$, $\hat a_{2\plusminus}$ under $\pi:P\rightarrow \partial \HK$,
and observe that $\hat l_{1\plus}\bsimeq s\hat l_{1\minus} s^{-1}$,
and by the third assertion of Lemma \ref{lm:arc_decomposition} and Lemma \ref{lm:kplus} and Corollary \ref{cor:kminus}, $\beta = 0$ implies that 
\[\partial M_n\bsimeq  l_2^{\lambda_\plus}\hat a_{2\plus}^\rho   \hat l_{1\plus}^{\mu_\plus+n}s \hat l_{1\minus}^{\mu_\minus}\hat a_{2\minus}^{-\rho} l_2^{\lambda_\minus+n}
\bsimeq
  l_2^{\lambda_\plus}\hat a_{2\plus}^\rho \hat l_{1\plus}^{\mu+n}s \hat a_{2\minus }^{-\rho} l_2^{\lambda_\minus+n}.
\] 
In particular, if $n=-\lambda$, then  
\[\partial M_{-\lambda}\bsimeq l_2^{\lambda_\plus}\hat a_{2\plus}^\rho \hat l_{1\plus}^{\mu-\lambda}s \hat a_{2\minus }^{-\rho} l_2^{-\lambda_\plus} 
.\]
Since $\hat a_{2\plus}^\rho\hat l_{1\plus}^{\mu-\lambda}s \hat a_{2\minus }^{-\rho}$ 
is disjoint from $d_2$, we can isotope $\partial M_{-\lambda}$, 
without fixing $b_2$, away from $d_2$. Particularly, 
it may be assumed that $\partial M_{-\lambda}\subset\partial V$, where $V\subset\HK$ 
is the solid torus cut off by the disk 
$D_2\subset\HK$ bounded by $d_2$. 

Set $\Delta:=\mu-\lambda$. Then we have $[\partial M_{-\lambda}]=\Theta[d_1]+(q\Delta+\delta)[l_0]\in H_1(\partial V)$. 
To determine $\Theta$, we note that
$\partial M_{-\lambda}$ meets $l_1$ 
positively at one point and 
$[l_1]=p[d_1]+q[l_0]\in H_1(\partial V)$.
In other words, $\Theta$ satisfies
\[
\begin{vmatrix}
\Theta & q\Delta+\delta\\
p      & q
\end{vmatrix}
=1,
\] 
and hence $\Theta=p\Delta+\frac{p\delta+1}{q}$.

\centerline{\bf Claim: $\vert \Delta\vert\leq 1$.} 
We prove by contradiction, that is, assuming $\vert\Delta\vert\geq 2$. 
Observe that if $\Delta\geq 2$, then 
\[q\Delta+\delta\geq \Delta+\delta\geq 2,\]
and if $\Delta\leq -2$, then 
\[q\Delta+\delta=q(\Delta+1)+\delta-q\leq -2\] 
since $\delta-q\leq -1$ and $q(\Delta+1)\leq -1$. 
This implies the boundary slope of $M_{-\lambda}$ 
is not integral with respect to $(\sphere, V)$, and hence $M_{-\lambda}$,  
as well as $A_{-\lambda}$, is inessential in $\Compl V$. Therefore
$A_{-\lambda}$ is of type $3$-$2$ii, and 
thus by Lemma \ref{lm:cri_threetwo_ii}, 
$(\sphere,cV)$ is trivial.
\cout{
Since $q\Delta+\delta\neq 0,\pm 1$, $M_{-\lambda}\subset \Compl{V}$ cannot be compressible, so the M\"obius band $M_{-\lambda}$ and hence the annulus
$A_{-\lambda}$ are $\partial$-compressible in $\Compl{V}$. 

Now, suppose $(\sphere,cV)$ 
is a non-trivial knot. 
Then the torus induced by 
the union of $A_{-\lambda}$ and the annulus cut off by $\partial A_{-\lambda}$ from $\partial\HK$ 
is incompressible, contradicting
the atoroidality of $\pair$.
Therefore $(\sphere, cV)$ is trivial. 
}
On the other hand, since $\vert q\Delta+\delta\vert\geq 2$, $M_{-\lambda}\subset \Compl{V}$ is incompressible, so $M_{-\lambda}$ is $\partial$-compressible in $\Compl{V}$, 
%
and the exterior $\Compl{V}$
can be regarded as a regular neighborhood of $M_{-\lambda}$. In particular, we have $\Theta=\pm 2$.
Set 
\[\Gamma:=q\Theta=pq\Delta+p\delta+1=
p(q\Delta+\delta)+1,\] 
and note that $\Gamma=\pm 2q$.
Recall that $\vert p\vert \geq 2$ by Lemma \ref{lm:typeK_non_separating} and $q>0, \delta\geq 0$ with $q\pm \delta\geq 1$
by Lemma \ref{lm:basic_loop}.

Suppose $\Delta>0$; therefore $\Delta\geq 2$ 
by the assumption.
Then 
$q\Delta+\delta\geq 2q+\delta.$
If $p\geq 2$, then  
\[\Gamma\geq p(2q+\delta)+1\geq 
2(2q+\delta)+1=2q+2(q+\delta)+1>2q.\]
If $p\leq -2$, then 
\[\Gamma\leq p(2q+\delta)+1\leq 
-2(2q+\delta)+1=-2q-2(q+\delta)+1\leq -2q-1<-2q.\]
Either case contradicts the fact $\Gamma= \pm 2q$, so $\Delta$ cannot be positive,
and hence  
$\Delta\leq -2$ by the assumption.
This implies
$q\Delta+\delta\leq -2q+\delta<0$.
If $p\leq -2$, then   
\[\Gamma\geq p(-2q+\delta)+1\geq 
-2(-2q+\delta)+1=2q+2(q-\delta)+1>2q.\]
If $p\geq 2$, then 
\[\Gamma\leq p(-2q+\delta)+1\leq 
2(-2q+\delta)+1=-2q-2(q-\delta)+1\leq -2q-1<-2q.\] 
In either case, we obtain $\Gamma\neq \pm 2q$, a contradiction.  
Therefore the claim.

\subsubsection*{Case $1.3$: $\beta=0, q\leq 2$.}
Recall that 
$[\partial M_n]$ is conjugate to
$v^{q(\mu+n)+\delta}u^{\lambda+n}$. 
Moreover, we have $\delta=1$ if $q=2$, and $\delta=0$ if $q=1$, and $\vert\Delta\vert =\vert\mu-\lambda\vert\leq 1$. 

By Lemmas \ref{lm:cri_fourone} and \ref{lm:cri_not_power_pimitive},
if $\lambda=\mu$, then
for any $n$ such that $n+\lambda\neq 0,\pm 1$, $A_n$ is of type $4$-$1$. 
If $\lambda=\mu+1$, then observe that $\lambda+n\neq 0,\pm 1$ 
implies $\mu+n\neq 0,-1$ or $-2$ and
$2(\mu+n)+1\neq 1,-1$ or $-3$. 
Therefore, when $q=2$   
(resp.\ $q=1$), for any $n$ such that 
$\lambda+n\neq 0,\pm 1$ (resp.\ 
and $\mu+n\neq 1$), $A_n$ 
is of type $4$-$1$.
Similarly, if $\lambda=\mu-1$, 
then    
$\lambda+n\neq 0,\pm 1$ 
implies $\mu+n \neq 0,1$ or $2$ and
$2(\mu+n)+1\neq 1,3$ or $5$. Thus,
whether $q=1$ or $2$, 
for any $n$ such that 
$\lambda+n\neq 0,\pm 1$ and
$q(\mu+n)+\delta\neq -1$, $A_n$ is of type $4$-$1$. Therefore, at most four $A_n$'s are not of type $4$-$1$ in this case.


\subsubsection*{Case $2$: $\beta\leq -1$.} 
By Lemma \ref{lm:homotopy_class_Mn}, 
$[\partial M_n]$ is conjugate to 
\[\mathcal{A}_{-\beta}(u,v^q)v^{q(\mu+n)+\delta}\mathcal{A}_{-\beta}(v^{-q},u^{-1})u^{\lambda+n}.\] 
By the second assertion of Lemma \ref{lm:arc_decomposition}, there exists a paired unit sequence 
$A'_{\beta'}$ with $\beta'=-\beta-1$ 
such that
\[
\mathcal{A}_{-\beta}(u,v^q)=u^{-1}\mathcal{A}'_{\beta'}(v^q,u)v^q,\text{ and } 
\mathcal{A}_{-\beta}(v^{-q},u^{-1})=v^q\mathcal{A}'_{\beta'}(u^{-1},v^{-q})u^{-1}.
\] 
Set $\mu':=\mu+2, \lambda':=\lambda-2$.
Then $[\partial M_n]$ is conjugate to
\begin{multline*}
\mathcal{A}'_{\beta'}(v^q,u)v^{q(\mu+2+n)+\delta}\mathcal{A}'_{\beta'}(u^{-1},v^{-q})u^{\lambda-2+n}\\=
\mathcal{A}'_{\beta'}(v^q,u)v^{q(\mu'+n)+\delta}\mathcal{A}'_{\beta'}(u^{-1},v^{-q})u^{\lambda'+n} 
\end{multline*} 
with $\beta'\geq 0$.
\subsubsection*{Case $2.1$: $\beta<-1$ or $\beta=-1,q>2$.} 
This corresponds to $\beta'>0$ or 
$\beta'=0, q>2$, so the same argument for Cases $1.1$-$2$ applies with $\beta,\lambda,\mu$ there replaced by $\beta',\lambda',\mu'$.
\subsubsection*{Case $2.2$: $\beta=-1, q\leq 2$.} 
By the third assertion of Lemma \ref{lm:arc_decomposition}, Lemma \ref{lm:kplus} and Corollary \ref{cor:kminus}, we have 
\[\partial M_n\bsimeq
  l_2^{-1+\lambda_\plus}\hat a_{2\plus}^{-\rho} \hat l_{1\plus}^{\mu+2+n}s \hat a_{2\minus }^\rho l_2^{-1+\lambda_\minus+n},\]
and hence 
\[\partial M_{-\lambda+2}\bsimeq
  l_2^{-1+\lambda_\plus}\hat a_{2\plus}^{-\rho} \hat l_{1\plus}^{\mu-\lambda+4}s \hat a_{2\minus }^\rho l_2^{1-\lambda_\plus}.\]
Applying the same argument preceding Case $1.3$ to $M_{-\lambda+2}$, we obtain $\vert \mu-\lambda+4\vert=\vert \mu'-\lambda'\vert\leq 1$. The same proof for Case $1.3$ then goes through 
with $\lambda,\mu$ there replaced with $\lambda',\mu'$. This completes the proof.
%
\end{proof}

\subsubsection{Example}\label{subsubsec:example}
Recall from \cite{Wan:23p} that 
the handlebody-knot $\pair$ in 
Figs.\ \ref{fig:example_A}, \ref{fig:example_M} is of type $\Kl$
and equivalent to the mirror of $5_2$ in \cite{IshKisMorSuz:12}. 
Let $A, \disk A, l_1\subset V_1,l_2\subset V_2, M_n, A_n, n\in\mathbb{Z}, k_\plusminus, b_2$ be as in Sections \ref{subsubsec:curves_nonchar} and 
\ref{subsubsec:orientation}. 
Denote by $u,v$ the generating pair 
of $\pi_1(\HK)$ given by the loops $l_2$, $k_\plus l_1 k_\plus^{-1}$ based at $b_2$, respectively. Then 
$\partial M_n$ determines an element $v^nu^{n+1}\in \pi_1(\HK)$, $n\in \mathbb{Z}$. 
By Lemma \ref{lm:cri_fourone} and \cite[Theorem $1$]{Zie:70}, $A_n$ is of type $4$-$1$ if and only if 
$n\neq 1,0,-1,-2$. Furthermore, by Lemma \ref{lm:cri_threetwo_ii}, the annulus $A_0$ is of type $3$-$2$ii  
since $\partial M_0$ is trivial, and $A_{1}$ is of type $3$-$2$i 
since the core of $M_1$ in $\sphere$ 
is a trefoil knot (see Fig.\ \ref{fig:example_core}). 
On the other hand,
there is an automorphism $f$ of $\pair$ (see Fig.\ \ref{fig:example_symmetry}) that swaps $l_1,l_2$, and it may be assumed that $f(A_n)=A_{-n-1}$. 
This implies $A_{-1},A_{-2}$
are of type $3$-$2$ii and type $3$-$2$i,
respectively. In particular, 
the core of $M_n$ is an Eudave-Mu\~noz knot when $n>1$ or $n<-2$; for instance, 
the core of $M_2$ is the $(-2,3,7)$-pretzel knot, the simplest hyperbolic knot with a non-integral toroidal surgery.  
\begin{figure}[h]
\begin{subfigure}[t]{.48\linewidth}
\centering
\begin{overpic}[scale=.15, percent] 
{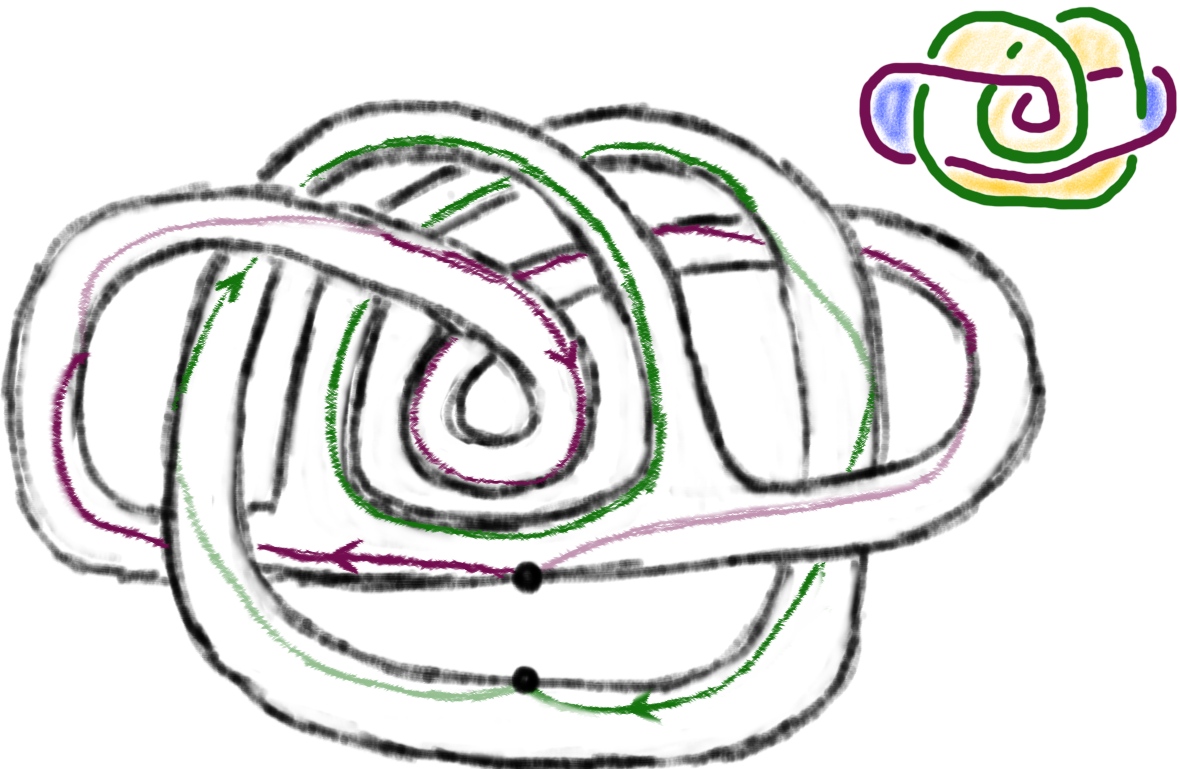}
\put(66,54){\large $A$:}
\put(43,2.8){\footnotesize $b_2$}
\put(42.5,12){\footnotesize $b_1$}
\put(65.5,7.5){$l_2$}
\put(25,20){$l_1$}
\end{overpic}
\caption{Type $3$-$3$ annulus $A$, $l_1,l_2,\disk A$.}
\label{fig:example_A}
\end{subfigure}
\begin{subfigure}[t]{.48\linewidth} 
\centering
\begin{overpic}[scale=.15, percent]{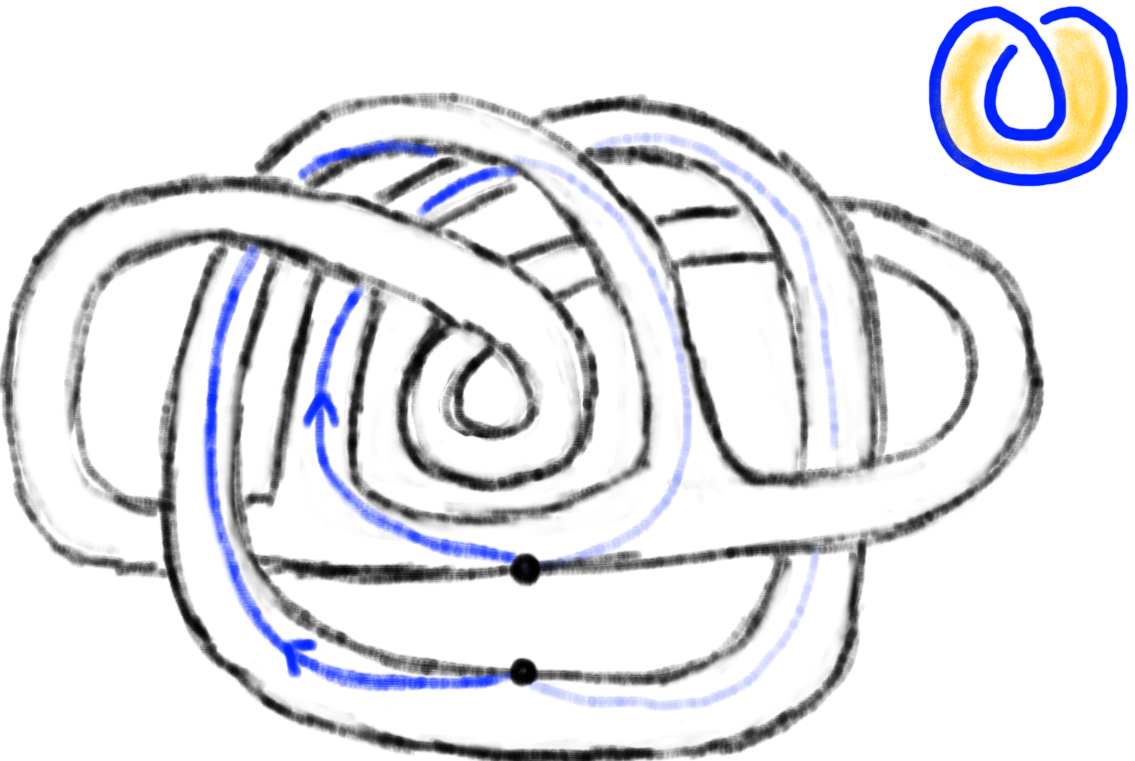}
\put(71,57){\large $M_0$:}
\put(46,3){\footnotesize $b_2$}
\put(46,12){\footnotesize $b_1$}
\put(19,9.5){\small $k_\plus$}
\put(25.5,22){\small $k_\minus$}
\end{overpic}
\caption{M\"obius band $M_0$, $k_\plusminus$.}
\label{fig:example_M}
\end{subfigure}
\begin{subfigure}[t]{.48\linewidth} 
\centering
\begin{overpic}[scale=.15, percent]{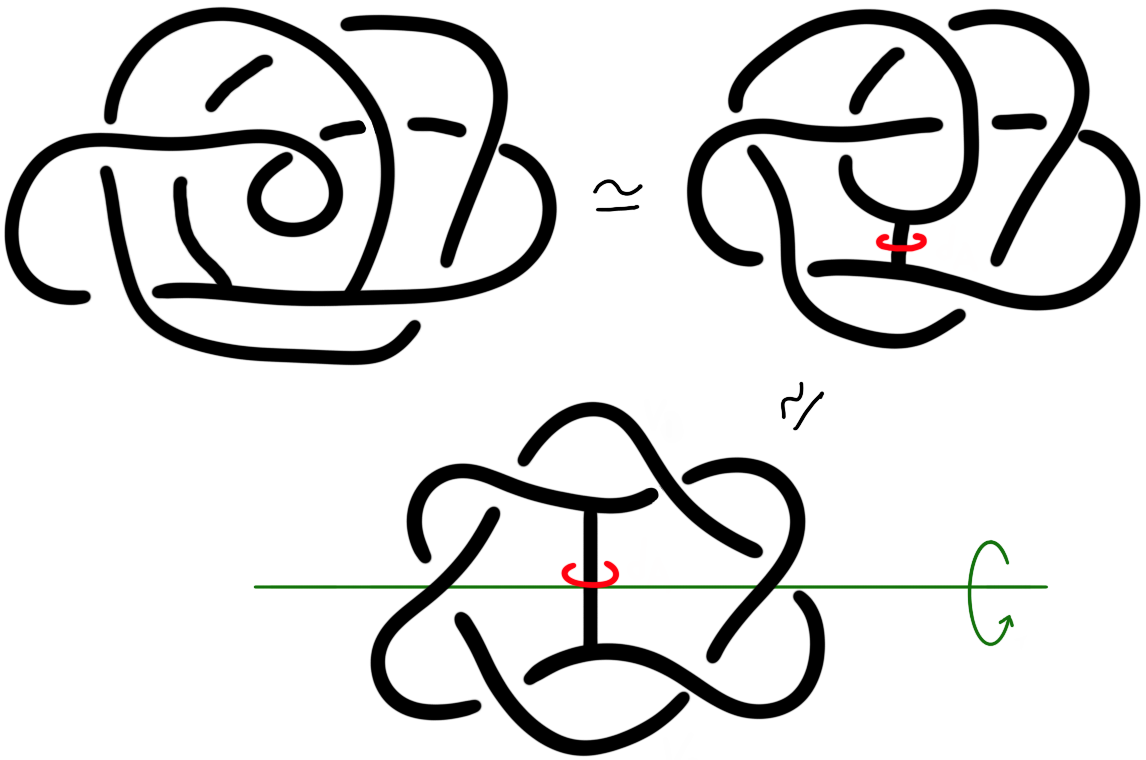}
\put(81.5,44.3){{\footnotesize $\disk A$}}
\put(54,17){$\disk A$}
\put(87,7.7){$\pi$}
\end{overpic} 
\caption{Automorphism.}
\label{fig:example_symmetry}
\end{subfigure}
\begin{subfigure}[t]{.48\linewidth} 
\centering
\begin{overpic}[scale=.15, percent]{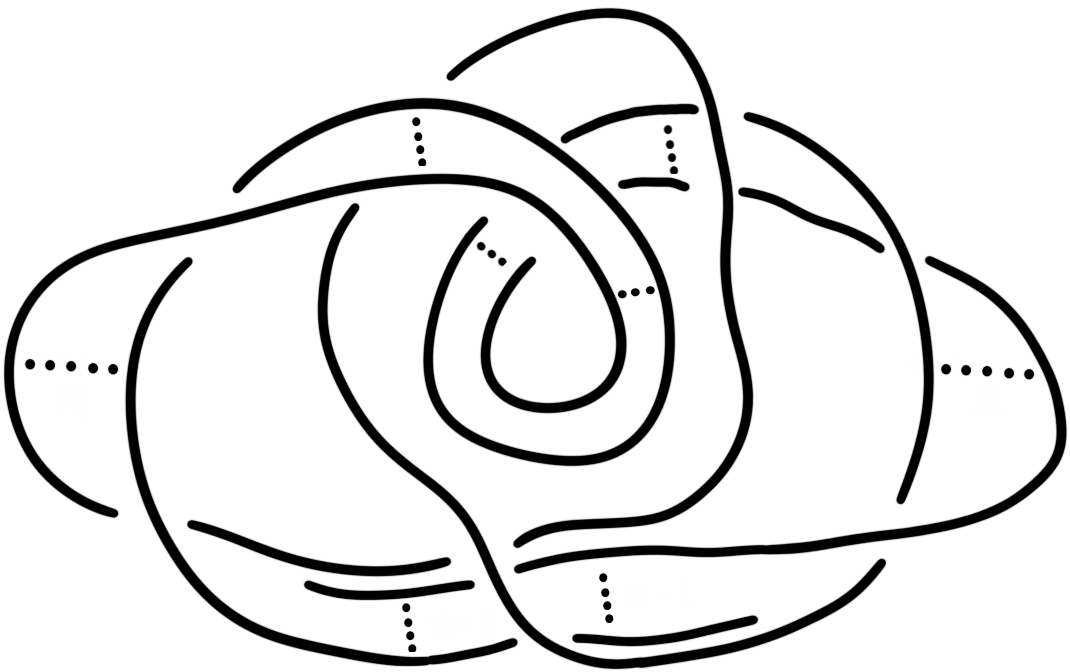}
\put(57.5,6.5){\footnotesize $n-1$} 
\put(27,4){\tiny $n-1$}
\put(90,24){$n$}
\put(5,24.5){$n$}
\end{overpic} 
\caption{Core of $M_n$, $n>0$.}
\label{fig:example_core}
\end{subfigure}
\caption{Example.}
\end{figure}


\section*{Acknowledgment}


\end{document}